\documentclass[10pt,english]{smfart}

\usepackage[T1]{fontenc}
\usepackage[english,french]{babel}
\usepackage{latexsym,amscd,color}
\usepackage{amsmath,amsfonts,amssymb,mathrsfs}
\usepackage{enumerate,euscript}
\usepackage{amssymb,url,xspace,smfthm}
\input xy
\xyoption{all}

\DeclareMathOperator{\gal}{Gal}

\newcommand{\BibTeX}{{\scshape Bib}\kern-.08em\TeX}

\newcommand{\T}{\S\kern .15em\relax }
\newcommand{\AMS}{$\mathcal{A}$\kern-.1667em\lower.5ex\hbox
        {$\mathcal{M}$}\kern-.125em$\mathcal{S}$}

\DeclareMathOperator{\im}{Im}

\DeclareMathOperator{\spm}{Spm}

\DeclareMathOperator{\rg}{rk}

\DeclareMathOperator{\spec}{Spec}

\newcommand{\wmu}{\widehat{\mu}}
\newcommand{\C}{\mathbb{C}}
\newcommand{\Q}{\mathbb{Q}}

\newcommand{\adeg}{\widehat{\deg}}

\newcommand{\p}{\mathfrak{p}}
\DeclareMathOperator{\sym}{Sym}

\newcommand{\E}{\overline{E}}
\newcommand{\F}{\overline{F}}
\newcommand{\sE}{\mathcal{E}}
\newcommand{\G}{\overline{G}}

\renewcommand{\O}{\mathcal{O}}
\renewcommand{\L}{\overline{L}}

\newcommand{\ndot}{\raisebox{.4ex}{.}}

\title{Explicit Uniform estimate of arithmetic Hilbert-Samuel function of hypersurfaces}
\alttitle{Estimation uniforme explicite de la fonction arithm\'etiques de Hilbert-Samuel des hypersurfaces}

\date{\today}
\author{Chunhui Liu}
\address{Department of Mathematics\\
Faculty of Science\\Kyoto University\\
606-8502 Kyoto\\Japan}
\email{chunhui.liu@math.kyoto-u.ac.jp}

\begin{document}
\def\smfbyname{}
\begin{abstract}
In this paper, we will give an upper bound and a lower bound of the arithmetic Hilbert-Samuel function of projective hypersurfaces, which are uniform and explicit. These two bounds have the optimal dominant terms. As an application, we use the lower bound to get an estimate of the density of rational points with small heights in a hypersurface.
\end{abstract}
\begin{altabstract}
Dans cet article, on donnera une majoration et une minoration de la fonction arithm\'etique de Hilbert-Samuel des hypersurfaces, qui sont uniformes et explicites. La majoration et la minoration admettent les termes principals optimaux. Comme une application, we obtient une estimation de la densit\'e des points rationnels de hauteur petite d'une hypersurface par cette minoration.
\end{altabstract}

\maketitle

\tableofcontents

\section{Introduction}
In this paper, we focus on an estimate of arithmetic Hilbert-Samuel function of arithmetic schemes. More precisely, we will give an upper bound and a lower bound of the arithmetic Hilbert-Samuel function of a hypersurface, which are both explicit and uniform.

\subsection{History}
Let $X$ be a closed sub-scheme of $\mathbb P^n_k$ over the field $k$ of dimension $d$, and $L$ be a very ample line bundle over $X$. We have (cf. \cite[Corollary 1.1.25, Theorem 1.2.6]{LazarsfeldI})
\[\dim_k\left(H^0(X,L^{\otimes D})\right)=\frac{\deg\left(c_1(L)^d\right)}{d!}D^d+o(D^{d})\]
for $D\in\mathbb N^+$. We call $\dim_k\left(H^0(X,L^{\otimes D})\right)$ the (geometric) Hilbert-Samuel function of $X$ with respect to $L$ of the variable $D\in\mathbb N^+$.

It is one of the central subjects in Arakelov geometry to find an arithmetic analogue of the Hilbert-Samuel function defined above. Let $K$ be a number field, $\O_K$ be its ring of integers, $M_{K,f}$ be the set of its finite places, and $M_{K,\infty}$ be the set of infinite places. We suppose that $\pi:\mathscr X\rightarrow \spec\O_K$ is a arithmetic variety of Krull dimension $d+1$, which means that $\mathscr X$ is integral and the morphism $\pi$ is flat and projective. Let $\overline{\mathscr L}=\left(\mathscr L,(\|\ndot\|_v)_{v\in M_{K,\infty}}\right)$ be a (semi-positive or positive) normed very ample line bundle over $\mathscr X$ (on the generic fiber). Let
\[\widehat{h}^0(\mathscr X,\overline{\mathscr L}^{\otimes D})=\log\#\left\{s\in H^0(\mathscr X,\mathscr L^{\otimes D})|\;\|s\|_v\leqslant1,\forall v\in M_{K,\infty}\right\}.\]
Or equivalently, we consider $H^0(\mathscr X,\mathscr L^{\otimes D})$ as a normed vector bundle equipped with some induced norms over $\spec\O_K$, and we consider the its (normalized) Arakelov degree $\adeg_n\left(\overline{H^0(\mathscr X,\mathscr L^{\otimes D})}\right)$ or its slope $\wmu\left(\overline{H^0(\mathscr X,\mathscr L^{\otimes D})}\right)$. Usually we call the above functions of $D\in\mathbb N^+$ the arithmetic Hilbert-Samuel function of $\left(\mathscr X,\overline{\mathscr L}\right)$.

We expect that we have
\[\widehat{h}^0(\mathscr X,\overline{\mathscr L}^{\otimes D})=\frac{\adeg\left(\widehat{c}_1(\overline{\mathscr L})^{d+1}\right)}{(d+1)!}D^{d+1}+o(D^{d+1})\]
for $D\in\mathbb N^+$, and for different cases, we have some results on this topic.

 In \cite{Gillet-Soule}, H. Gillet and C. Soul\'e proved such an asymptotic formula (\cite[Theorem 8]{Gillet-Soule}) with respect to a Hermitian line bundle as a consequence of the arithmetic Riemann-Roch theorem (\cite[Theorem 7]{Gillet-Soule}), where they suppose $\mathscr X$ has a regular generic fiber. In \cite[Th\'eor\`eme principal]{Abbes-Bouche}, A. Abbes and T. Bouche proved the same result without the arithmetic Riemann-Roch theorem supposing the same condition on the generic fiber. In \cite[Theorem 1.4]{Zhang95}, S. Zhang proved this result without the condition of smooth generic fiber by some technique of asymptotic analysis. In \cite[Th\'eor\`eme A]{Randriam06}, H. Randriambololona generalized this result to the case of coherent sheaf from a sub-quotient of the normed vector bundle.

In \cite{Philippon_Sombra_2008}, P. Philippon et M. Sombra proposed another definition of the arithmetic Hilbert-Samuel function, and they proved an asymptotic formula for the case of toric varieties (see \cite[Th\'eor\`eme 0.1]{Philippon_Sombra_2008}). In \cite{Hajli_2015}, M. Hajli proved the same asymptotic formula for the case of general projective varieties with the definition in \cite{Philippon_Sombra_2008}.

It is also an important topic to consider the uniform bounds of the arithmetic Hilbert-Samuel function of arithmetic varieties, with a possibly minor modification of the definition. In \cite{David_Philippon99}, S. David and P. Philippon give an explicit uniform lower bound of the the arithmetic Hilbert-Samuel function. This result is reformulated by H. Chen in \cite[Theorem 4.8]{Chen1} for a study of counting rational points. In fact, let $\mathscr X\rightarrow\spec\O_K$ be an arithmetic varieties, and $\overline{\mathscr L}$ be a very ample Hermitian line bundle over $\mathscr X$ which determines a polarization in $\mathbb P^n_{\O_K}$ such that $\deg\left(\mathscr X\times_{\spec\O_K}\spec K\right)=\delta$ as a closed sub-scheme of $\mathbb P^n_K$. We denote by $\wmu(\F_D)$ the arithmetic Hilbert-Samuel function the above $\mathscr X$ (see Definition \ref{arithmetic hilbert function} for the precise definition) of variable $D\in\mathbb N^+$. In the above literatures, the inequality
\[\frac{\wmu(\F_D)}{D}\geqslant\frac{d!}{\delta(2d+2)^{d+1}}h_{\overline{\mathscr L}}(\mathscr X)-\log(n+1)-2^d\]
is uniformly verified for any $D\geqslant2(n-d)(\delta-1)+d+2$ (see also \cite[Remark 4.9]{Chen1} for some minor modification), where $h_{\overline{\mathscr L}}(\mathscr X)$ is the height of $\mathscr X$ defined by the arithmetic intersection theory (cf. \cite[Definition 2.5]{Faltings91}). But this estimate is far from optimal. Even the coefficient of $h_{\overline{\mathscr L}}(\mathscr X)$ is not optimal compared with that in the asymptotic formula.
\subsection{Principle result}
In this paper, we will give an upper bound and a lower bound of the arithmetic Hilbert-Samuel function of a hypersurface. In fact, we will prove the following result (in Theorem \ref{upper and lower bound of arithmetic Hilbert}).
\begin{theo}\label{upper and lower bound of arithmetic Hilbert-intro}
  Let $\O(1)$ be the universal bundle of $\mathbb P^n_K$ equipped with $\ell^2$-nomrs (see \eqref{l^2-norm} for the definition). Let $X$ be a hypersurface of degree $\delta$ in $\mathbb P^n_K$. We denote by $\wmu(\F_D)$ (see Definition \ref{arithmetic hilbert function} for the precise definition of the Hermitian vector bundle $\F_D$ over $\spec\O_K$) the arithmetic Hilbert-Samuel function of variable $D$, which is defined with respect to $\O(1)$ and the above closed immersion. Then the inequalities
\[\frac{\wmu(\F_D)}{D}\geqslant\frac{h(X)}{n\delta}+C_1(n,\delta)\]
and
\[\frac{\wmu(\F_D)}{D}\leqslant\frac{h(X)}{n\delta}+C_2(n,\delta)\]
are uniformly verified for all $D\geqslant\delta$, where $h(X)$ is a logarithmic height of $X$ (Definition \ref{classical height of hypersurface} or Definition \ref{definition of wmu(I_X)}), and the constants $C_1(n,\delta)$ and $C_2(n,\delta)$ will be given explicitly in Theorem \ref{upper and lower bound of arithmetic Hilbert}.
\end{theo}
Since we can explicitly compare the possible involved heights of $X$ (see \cite[\S3,\S4]{BGS94}, \cite[\S B.7]{Hindry}, \cite[Proposition 3.6]{Chen1} and Proposition \ref{height na\"ive-slope} of this paper), it is not very serious to choose what kind of heights of $X$ in the statement of Theorem \ref{upper and lower bound of arithmetic Hilbert-intro} if we do not care the constants $C_1(n,\delta)$ and $C_2(n,\delta)$ above. In Theorem \ref{upper and lower bound of arithmetic Hilbert-intro}, the coefficients of $h(X)$ are optimal compared with that in the asymptotic formula. In fact, we consider a special case in this paper, but we get a better estimate than that in \cite{David_Philippon99} and \cite[Theorem 4.8]{Chen1}.
\subsection{Motivation and an application}
The lower bound of the arithmetic Hilbert-Samuel function can be applied in the problem of counting rational points by the determinant method, see \cite[Theorem 3.1]{Chen2} for example. By \cite[Proposition 2.12]{Chen1}, if the heights of rational points can be bounded by a term depending on the arithmetic Hilbert-Samuel function, then these rational points can be covered by a hypersurface which does not contain the generic point of the original variety. In order to apply this proposition, a good uniform lower bound of the arithmetic Hilbert-Samuel function plays an important role. In Theorem \ref{covered by one hypersurface} (see also Remark \ref{explicit of covered by one hypersurface}), we will prove the following result.
\begin{theo}\label{covered by one hypersurface-intro}
  Let $K$ be a number field, $X$ be an integral hypersurface in $\mathbb P^n_K$ of degree $\delta$, and $H_K(X)$ be the height of $X$ (see Definition \ref{classical height of hypersurface} for the precise definition). We suppose that $S(X;B)$ is the set of rational points of $X$ whose heights (see Definition \ref{weil height} and Definition \ref{arakelov height}) are smaller than or equal to $B$. If
  \[H_K(X)\gg_{n,K,\delta}B^{n\delta},\]
  then $S(X;B)$ can be covered by one hypersurface of degree smaller than or equal to $\delta$ which does not contain the generic point of $X$.
\end{theo}
The implicit constant depending on $n$ and $K$ in Theorem \ref{covered by one hypersurface-intro} will be explicitly written down in Theorem \ref{covered by one hypersurface}.

These kinds of estimates are useful in the determinant method (see \cite{Heath-Brown,Salberger07,Salberger_preprint2013} for the classical version and \cite{Chen1,Chen2} for the approach of Arakelov geometry). In \cite[Theorem 4]{Heath-Brown}, \cite[Lemma 3]{Browning_Heath05}, \cite[Lemma 6.3]{Salberger07}, and \cite[Lemma 1.7]{Salberger_preprint2013}, the exponent of $B$ in Theorem \ref{covered by one hypersurface-intro} is of $O_{n}(\delta^3)$. In the remark under the statement of \cite[Theorem 1.3]{Walsh_2015}, the same exponent of $B$ as in Theorem \ref{covered by one hypersurface-intro} is obtained, but the degree of the auxiliary hypersurface can only be given as $O_{n,\delta}(1)$. All the arguments mentioned above work on $K=\Q$ and their constants are implicit.
\subsection{Structure of article}
This paper is organized as following: in \S 2, we provide the basic setting of the whole problem. In \S 3, first we state an estimate of arithmetic Hilbert-Samuel function of projective spaces, and we get the estimates of that of hypersurfaces by comparing some norms. In \S 4, by applying the slope inequalities over an evaluation map, we prove that the rational points of small heights in a hypersurface can be covered by a bounded degree hypersurface which does not contain its generic points.

In the appendix, we will give a uniform explicit estimate of the arithmetic Hilbert-Samuel function of projective spaces with respect to the symmetric norm, which is finer than that in \cite[Annexe]{Gaudron08}.
\subsection*{Acknowledgement}
This paper is a part of the author's Ph. D. thesis prepared at Universit\'e Paris Diderot - Paris 7. The author would like to thank one of his advisors Huayi Chen for his suggestions and kind-hearted help on this work. In addition, the author would like to thank Yeping Zhang for his aid in some calculation in the appendix. Chunhui Liu is supported by JSPS KAKENHI Grant Number JP17F17730.
\section{Notations and Preliminaries}
In this section, we will introduce some useful notations and definitions. Let $K$ be a number field, and $\O_K$ be its ring of integers. We denote by $M_{K,f}$ the set of finite places of $K$, and by $M_{K,\infty}$ the set of infinite places of $K$. In addition, we denote by $M_K=M_{K,f}\sqcup M_{K,\infty}$ the set of places of $K$. For every $v\in M_K$, we define the absolute value $|x|_v=\left|N_{K_v/\Q_v}(x)\right|_v^\frac{1}{[K_v:\Q_v]}$ for each $v\in M_K$, extending the usual absolute values on $\Q_p$ or $\mathbb{R}$.
\subsection{Hermitian vector bundles}
  A \textit{Hermitian vector bundle} over $\spec\O_K$ is all the pairings $\E=\left(E,\left(\|\ndot\|_v\right)_{v\in M_{K,\infty}}\right)$, where:
  \begin{itemize}
    \item $E$ is a projective $\O_K$-module of finite rang;
    \item $\left(\|\ndot\|_v\right)_{v\in M_{K,\infty}}$ is a family of Hermitian norms, where $\|\ndot\|_v$ is a norm over $E\otimes_{\O_K,v}\C$ which is invariant under the action of $\gal(\C/K_v)$.
  \end{itemize}

If $\rg_{\O_K}(E)=1$, we say that $\E$ is a \textit{Hermitian line bundle}.

We suppose that $F$ is a sub-$\O_K$-module of $E$. We say that $F$ is a \textit{saturated} sub-$\O_K$-module if $E/F$ is a torsion-free $\O_K$-module.

Let $\E=\left(E,\left(\|\ndot\|_{E,v}\right)_{v\in M_{K,\infty}}\right)$ and $\F=\left(F,\left(\|\ndot\|_{F,v}\right)_{v\in M_{K,\infty}}\right)$ be two Hermitian vector bundles. If $F$ is a saturated sub-$\O_K$-module of $E$ and $\|\ndot\|_{F,v}$ is the restriction of $\|\ndot\|_{E,v}$ over $F\otimes_{\O_K,v}\C$ for every $v\in M_{K,\infty}$, we say that $\F$ is a \textit{sub-Hermitian vector bundle} of $\E$ over $\spec\O_K$.

We say that $\G=\left(G,\left(\|\ndot\|_{G,v}\right)_{v\in M_{K,\infty}}\right)$ is a \textit{quotient Hermitian vector bundle} of $\E$ over $\spec\O_K$, if $G$ is a projective quotient $\O_K$-module of $E$ and $\|\ndot\|_{G,v}$ is the induced quotient space norm of $\|\ndot\|_{E,v}$ for every $v\in M_{K,\infty}$.

With the above definitions. If
\[\begin{CD}
  0@>>>F@>>>E@>>>G@>>>0
\end{CD}\]
is a exact sequence of $\O_K$-modules, we say that
\[\begin{CD}
  0@>>>\F@>>>\E@>>>\G@>>>0
\end{CD}\]
is a exact sequence of Hermitian vector bundles over $\spec\O_K$, and we denote $\G=\E/\F$.

For simplicity, we denote by $E_K=E\otimes_{\O_K}K$ in this paper below.

\subsection{Arakelov degree and slope}
Let $\E$ be a Hermitian vector bundle over $\spec\O_K$, and $\{s_1,\ldots,s_r\}$ be a $K$-basis of $E\otimes_{\O_K}K$. The \textit{Arakelov degree} of $\E$ is defined as
\begin{eqnarray*}
  \adeg(\E)&=&-\sum_{v\in M_{K}}[K_v:\Q_v]\log\left\|s_1\wedge\cdots\wedge s_r\right\|_v\\
  &=&\log\left(\#\left(E/\O_Ks_1+\cdots+\O_Ks_r\right)\right)-\frac{1}{2}\sum_{v\in M_{K,\infty}}\log\det\left(\langle s_i,s_j\rangle_{1\leqslant i,j\leqslant r}\right),
\end{eqnarray*}
where $\left\|s_1\wedge\cdots\wedge s_r\right\|_v$ follows the definition in \cite[2.1.9]{Chen10b} for all $v\in M_{K,\infty}$, and $\langle s_i,s_j\rangle_{1\leqslant i,j\leqslant r}$ is the Gram matrix of the basis $\{s_1,\ldots,s_r\}$. We refer the readers to \cite[2.4.1]{Gillet-Soule91} for a proof of the equivalence of the above two definitions. The Arakelov degree is independent of the choice of the basis $\{s_1,\ldots,s_r\}$ by the product formula (cf. \cite[Chap. III, Proposition 1.3]{Neukirch}). In addition, we define
\[\adeg_n(\E)=\frac{1}{[K:\Q]}\adeg(\E)\]
as the \textit{normalized Arakelov degree} of $\E$, which is independent of the choice of the base field $K$.

Let $\E$ be a non-zero Hermitian vector bundle over $\spec\O_K$, and $\rg(E)$ be the rank of $E$. The \textit{slope} of $\E$ is defined as
\[\wmu(\E):=\frac{1}{\rg(E)}\adeg_n(\E).\]
In addition, we denote by $\wmu_{\max}(\E)$ the maximal value of slopes of all non-zero Hermitian sub-bundles, and by $\wmu_{\min}(\E)$ the minimal value of slopes of all non-zero Hermitian quotients bundles of $\E$.

Let
\[\begin{CD}
  0@>>>\F@>>>\E@>>>\G@>>>0
\end{CD}\]
be an exact sequence of Hermitian vector bundles. In this case, we have
\begin{equation}\label{Po}
  \adeg(\E)=\adeg(\F)+\adeg(\G).
\end{equation}
We refer the readers to \cite[(3.3)]{Bost_Kunnemann} for a proof of the equality \eqref{Po}.
\subsection{Definition of arithmetic Hilbert-Samuel function}\label{basic setting}
Let $\overline{\mathcal E}$ be a Hermitian vector bundle of rank $n+1$ on $\spec\O_K$, and $\mathbb P(\sE)$ be the projective space which represents the functor from the category of commutative $\O_K$-algebras to the category of sets mapping all $\O_K$-algebra $A$ to the set of rank $1$ projective quotient $A$-module of $\sE\otimes_{\O_K}A$. We denote by $\pi:\mathbb P(\sE)\rightarrow\spec\O_K$ the structural morphism. Let $\O_{\mathbb P (\sE)}(1)$ (or by $\O(1)$ if there is no confusion) be the universal bundle, and $\O_{\mathbb P (\sE)}(D)$ (or $\O(D)$) be the line bundle $\O_{\mathbb P (\sE)}(1)^{\otimes D}$. The Hermitian metrics on $\sE$ induce by quotient of Hermitian metrics (i.e. Fubini-Study metrics) on $\O_{\mathbb P(\sE)}(1)$ which define a Hermitian line bundle $\overline{\O_{\mathbb P(\sE)}(1)}$ on $\mathbb P(\sE)$.

For every $D\in\mathbb N^+$, let
\begin{equation}\label{definition of E_D}
  E_D=H^0\left(\mathbb P(\sE),\O_{\mathbb P (\sE)}(D)\right)
\end{equation} and let $r(n,D)$ be the its rank over $\O_K$. In fact, we have
\begin{equation}\label{def of r(n,D)}
  r(n,D)={n+D\choose D}.
\end{equation}
For each $v\in M_{K,\infty}$, we denote by $\|\ndot\|_{v,\sup}$ the norm over $E_{D,v}=E_D\otimes_{\O_K,v}\C$ such that
\begin{equation}\label{definition of sup norm}
  \forall s\in E_{D,v},\;\|s\|_{v,\sup}=\sup_{x\in\mathbb P(\sE_K)_v(\C)}\|s(x)\|_{v,\mathrm{FS}},
\end{equation}
where $\|\ndot\|_{v,\mathrm{FS}}$ is the corresponding Fubini-Study norm.

Next, we will introduce the \textit{metric of John}, see \cite{Thompson96} for a systematic introduction of this notion. In general, for a given symmetric convex body $C$, there exists the unique ellipsoid $J(C)$, called \textit{ellipsoid of John}, contained in $C$ whose volume is maximal.

For the $\O_K$-module $E_D$ and any place $v\in M_{K,\infty}$, we take the ellipsoid of John of its unit closed ball defined via the norm$\|\ndot\|_{v,\sup}$, and this ellipsoid induces a Hermitian norm, noted by $\|\ndot\|_{v,J}$ this norm. For every section $s\in E_{D}$, the inequality
\begin{equation}\label{john norm}
  \|s\|_{v,\sup}\leqslant\|s\|_{v,J}\leqslant\sqrt{r(n,D)}\|s\|_{v,\sup}
\end{equation}
is verified from \cite[Theorem 3.3.6]{Thompson96}. In fact, these constants do not depend on the choice of the symmetric convex body.

By the estimate \eqref{john norm}, we have the proposition below.
\begin{prop}[\cite{Chen10b}, Proposition 2.1.14]\label{slope of different norms}
  Let $\E$ be a Hermitian vector bundle of rank $r>0$ over $\spec \O_K$. Then we have the following inequalities:
\begin{equation*}
  \wmu(\E)-\frac{1}{2}\log r\leqslant\wmu(\E_J)\leqslant\wmu(\E),
  \end{equation*}
  \begin{equation*}
  \wmu_{\max}(\E)-\frac{1}{2}\log r\leqslant\wmu_{\max}(\E_J)\leqslant\wmu_{\max}(\E),
\end{equation*}
  \begin{equation*}
  \wmu_{\min}(\E)-\frac{1}{2}\log r\leqslant\wmu_{\min}(\E_J)\leqslant\wmu_{\min}(\E),
\end{equation*}
where $\E_J$ is the Hermitian vector bundle equipped withe the norms of John induced from the original norms.
\end{prop}
\begin{rema}
The reason why we consider the metric of John is that the hermitian vector bundles equipped with this metrics are useful in the determinant method reformulated by Arakelov geometry, see \cite{Chen1,Chen2} for more details about this. This method is important and useful in counting rational points problems.
\end{rema}
Let $A$ be a ring, and $E$ be an $A$-module. We denote by $\sym^D_{A}(E)$ the symmetric product of degree $D$ of the $A$-module $E$, or by $\sym^D(E)$ if there is no confusion on the base ring.

If we consider the above $E_D$ defined in \eqref{definition of E_D} as a $\O_K$-module, we have the isomorphism of $\O_K$-modules $E_D\cong \sym^D(\mathcal{E})$. Then for every place $v\in M_{K,\infty}$, the Hermitian norm $\|\ndot\|_v$ over $\mathcal{E}_{v,\C}$ induces a Hermitian norm $\|\ndot\|_{v,\mathrm{sym}}$ by the symmetric product. More precisely, this norm is the quotient norm induced by the quotient morphism
\[\sE^{\otimes D}\rightarrow\sym^D(\sE),\]
 where the vector bundle $\overline{\sE}^{\otimes D}$ is equipped with the norm of tensor product of $\overline{\sE}$ over $\spec\O_K$ (cf. \cite[D\'efinition 2.10]{Gaudron08} for a definition). We say that this norm is the \textit{symmetric norm} over $\sym^D(\sE)$. For any place $v\in M_{K,\infty}$, the $\|\ndot\|_{v,J}$ and $\|\ndot\|_{v,\mathrm{sym}}$ are invariant under the action of the unitary group $U(\sE_{v,\C},\|\ndot\|_v)$ of order $n+1$. Then they are proportional and the ratio is independent of the choice of $v\in M_{K,\infty}$ (see \cite[Lemma 4.3.6]{BGS94} for a proof). We denote by $R_0(n,D)$ the constant such that, for every section $0\neq s\in E_{D,v}$, the equality
\begin{equation}\label{symmetric norm vs John norm}
  \log\|s\|_{v,J}=\log\|s\|_{v,\mathrm{sym}}+R_0(n,D).
\end{equation}
is verified.
\begin{defi}\label{definition of E_D with norm}
Let $E_D$ be the $\O_K$-module defined in \eqref{definition of E_D}. For every place $v\in M_{K,\infty}$, we denote by $\E_D$ the Hermitian vector bundle over $\spec\O_K$ which $E_D$ is equipped with the norm of John $\|\ndot\|_{v,J}$ induced by the norms $\|\ndot\|_{v,\sup}$ defined in \eqref{definition of sup norm}. Similarly, we denote by $\E_{D,\mathrm{sym}}$ the Hermitian vector bundle over $\spec\O_K$ which $E_D$ is equipped with the norms $\|\ndot\|_{v,\mathrm{sym}}$ introduced above.
\end{defi}
With all the notations in Definition \ref{definition of E_D with norm}, we have the following result.
\begin{prop}[\cite{Chen1}, Proposition 2.7]\label{symmetric norm vs John norm, constant}
  With all the notations in Definition \ref{definition of E_D with norm}, we have
\[\wmu_{\min}(\E_D)=\wmu_{\min}(\E_{D,\mathrm{sym}})-R_0(n,D).\]
In the above equality, the constant $R_0(n,D)$ defined in the equality \eqref{symmetric norm vs John norm} satisfies the inequality
\begin{equation*}
  0\leqslant R_0(n,D)\leqslant\log\sqrt{r(n,D)},
\end{equation*}
where the constant $r(n,D)=\rg(E_D)$ follows the definition in the equality \eqref{def of r(n,D)}.
\end{prop}

Let $X$ be a pure dimensional closed sub-scheme of $\mathbb{P}(\mathcal{E}_K)$, and $\mathscr{X}$ be the Zariski closure of $X$ in $\mathbb{P}(\mathcal{E})$. We denote by
\begin{equation}\label{evaluation map}
\eta_{X,D}:\;E_{D,K}=H^0\left(\mathbb{P}(\mathcal{E}_K),\O(D)\right)\rightarrow H^0\left(X,\O_{\mathbb P(\sE_K)}(1)|_X^{\otimes D}\right)
\end{equation}
the \textit{evaluation map} over $X$ induced by the closed immersion of $X$ in $\mathbb P(\sE_K)$. We denote by $F_D$ the saturated image of the morphism $\eta_{X,D}$ in $H^0\left(\mathscr{X},\O_{\mathbb P(\sE)}(1)|_\mathscr{X}^{\otimes D}\right)$. In other words, the $\O_K$-module $F_D$ is the largest saturated sub-$\O_K$-module of $H^0\left(\mathscr{X},\O_{\mathbb P(\sE)}(1)|_\mathscr{X}^{\otimes D}\right)$ such that $F_{D,K}=\im(\eta_{X,D})$. When the integer $D$ is large enough, the homomorphism $\eta_{X,D}$ is surjective, which means $F_D=H^0(\mathscr{X},\O_{\mathbb P(\sE)}(1)|_\mathscr{X}^{\otimes D})$.

The $\O_K$-module $F_D$ is equipped with the quotient metrics (from $\E_D$) such that $F_D$ is a Hermitian vector bundle over $\spec \O_K$, noted by $\F_D$ this Hermitian vector bundle.
\begin{defi}[Arithmetic Hilbert-Samuel function]\label{arithmetic hilbert function}
Let $\F_D$ be the Hermitian vector bundle over $\spec\O_K$ defined above from the map \eqref{evaluation map}. We say that the function which maps the positive integer $D$ to $\wmu(\F_D)$ is the \textit{arithmetic Hilbert-Samuel function} of $X$ with respect to the Hermitian line bundle $\overline{\O(1)}$.
\end{defi}
\begin{rema}
  With all the notations in Definition \ref{arithmetic hilbert function}. By \cite[Th\'eor\`eme A]{Randriam06}, we have
  \[h_{\overline{\O(1)}}(X)=\lim_{D\rightarrow+\infty}\frac{\adeg_n(\F_D)}{D^{d+1}/(d+1)!},\]
  where the height $h_{\overline{\O(1)}}(X)=\adeg_n\left(\widehat{c}_1\left(\overline{\O(1)}\right)^{d+1}\cdot\left[\mathscr X\right]\right)$ of $X$ is defined by the arithmetic intersection theory (cf. \cite[Definition 2.5]{Faltings91}).
\end{rema}
\subsection{A height function of projective varieties}
In this part, we will introduce a kind of height functions of a projective scheme $X$. For this aim, first we will introduce the notion of Cayley form of the scheme $X$, where we follow the construction in \cite[\S3]{Chen1}. Next, we will construct a family of generators of $X$ from its Cayley form, and we define this height of $X$ as the slope of the Hermitian vector bundle generated by these generators equipped with some induced norms.

Cayley form has a very close relation to Chow form, which is also introduced in the above reference. In this paper, the notion of Chow form will not be used.
\subsubsection*{Cayley forms}
Let $\overline{\sE}=\left(\sE,\left(\|\ndot\|_v\right)_{v\in M_{K,\infty}}\right)$ be a Hermitian vector bundle of rank $n+1$ over $\spec\O_K$. Let $\check{G}=\hbox{Gr}\left(d+1,\sE_K^\vee\right)$ be the Grassmannian which classifies all the quotients of rank $d+1$ of $\sE_K^\vee$ (or we can say all the sub-space of rank $d+1$ of $\sE_K$). By the Pl\"{u}cker embedding $\check{G}\rightarrow\mathbb P\left(\wedge^{d+1}\sE_K^\vee\right)$, the coordinate algebra $B(\check{G})=\bigoplus\limits_{D\geqslant 0}B_D(\check{G})$ of $\check{G}$ is a homogeneous quotient algebra of the algebra $\bigoplus\limits_{D\geqslant 0}\sym_K^D\left(\wedge^{d+1}\sE_K^\vee\right)$. To explain the role of Pl\"{u}cker coordinate, we consider the following construction: we denote by
\[\theta:\;\sE_K^\vee\otimes_k(\wedge^{d+1}\sE_K)\rightarrow\wedge^d\sE_K\]
 the homomorphism which maps $\xi\otimes(x_0\wedge\cdots\wedge x_d)$ to
\begin{equation*}
  \sum_{i=0}^d(-1)^i\xi(x_i)x_0\wedge\cdots\wedge x_{i-1}\wedge x_{i+1}\wedge\cdots\wedge x_d.
\end{equation*}
 Let $\widetilde{\Gamma}$ be the sub-scheme of $\mathbb P(\sE_K)\times_K \mathbb P\left(\wedge^{d+1}\sE_K^\vee\right)$ which classifies the point $(\xi,\alpha)$ such that $\theta(\xi\otimes\alpha)=0$. Let $\widetilde{p}: \mathbb P(\sE_K)\times_K\mathbb P\left(\wedge^{d+1}\sE_K^\vee\right)\rightarrow\mathbb P(\sE_K)$ and $\widetilde{q}: \mathbb P(\sE_K)\times_K\mathbb P\left(\wedge^{d+1}\sE_K^\vee\right)\rightarrow\mathbb P\left(\wedge^{d+1}\sE_K^\vee\right)$ be the two canonical projections. Then we have the following proposition.
\begin{prop}[Proposition 2.6, \cite{Liu-non_reduced}]\label{cayleyform}
   Let $X$ be a pure dimensional closed sub-scheme of $\mathbb P(\sE_K)$, which is of dimension $d$ and of degree $\delta$. We suppose that $[X]=\sum\limits_{i\in I}m_iX_i$ is the fundamental cycle (see \cite[\S 1.5]{Fulton} for the definition) of $X$. Then $\widetilde{q}\left(\widetilde{\Gamma}\cap \widetilde{p}^{-1}(X)\right)$ is a hypersurface of degree $\delta$ of $\mathbb P\left(\wedge^{d+1}\sE_K^\vee\right)$. In addition, the fundamental cycle of the hypersurface $\widetilde{q}\left(\widetilde{\Gamma}\cap \widetilde{p}^{-1}(X)\right)$ has the form of
  \[\sum_{i\in I}m_i\widetilde{X}'_i,\]
  where all these $\widetilde{X}'_i$ are distinct integral hypersurfaces of degree $\deg(X_i)$ in $\mathbb P\left(\wedge^{d+1}\sE_K^\vee\right)$ respectively.
\end{prop}

\begin{defi}\label{cayley form 2}
We denote by $\Psi_{X,K}$ the sub-$K$-vector space of dimension $1$ of $\sym_K^\delta\left(\wedge^{d+1}\sE_K^\vee\right)$ which defines the hypersurface determined in Proposition \ref{cayleyform}. We call it the \textit{Cayley form} of $X$. The incidence variety $\Gamma'$ of $\mathbb P(\sE_K)\times_K\check{G}$ is the intersection of $\widetilde{\Gamma}$ and $\mathbb P(\sE_K)\times_K\check{G}$ (embedded in $\mathbb P(\sE_K)\times_K\mathbb P\left(\wedge^{d+1}\sE_K^\vee\right)$).
\end{defi}

We will construct a system of generators of $X$ from $\Psi_{X,K}$ which is of degree $\delta$, where we follow the construction of \cite[\S 3]{Chen1}. Let $\psi_X\in\sym_K^\delta\left(\wedge^{d+1}\sE_K^\vee\right)$ be a non-zero element which represents the Cayley form $\Psi_{X,K}$ of X. This element is considered to be a homogeneous polynomial of degree $\delta$ in $\wedge^{d+1}\mathcal E_K^{\vee}$. Let $x,y_0,\ldots,y_d$ be the variables in $\mathcal E_K$ and $\xi$ be a variable in $\mathcal E_K^\vee$. For any $i=0,1,\ldots,d$, let $z_i=\xi(x)y_i-\xi(y_i)x$. Since
\begin{eqnarray*}
  & &z_0\wedge\cdots\wedge z_d\\
&=&\xi(x)^{d+1}y_0\wedge\cdots\wedge y_d-\sum_{i=0}^d\xi(x)^d\xi(y_i)y_0\wedge\cdots\wedge y_{i-1}\wedge x\wedge y_{i+1}\wedge\cdots\wedge y_d\\
&=&\xi(x)^d\left(\xi(x)y_0\wedge\cdots\wedge y_d-\sum_{i=0}^d\xi(y_i)y_0\wedge\cdots\wedge y_{i-1}\wedge x\wedge y_{i+1}\wedge\cdots\wedge y_d\right),
\end{eqnarray*}
so we have
\begin{eqnarray*}
 & & \psi_{X,K}(z_0\wedge\cdots\wedge z_d)\\
&=&\xi(x)^{\delta d}\psi_{X,K}\left(\xi(x)y_0\wedge\cdots\wedge y_d-\sum_{i=0}^d\xi(y_i)y_0\wedge\cdots\wedge y_{i-1}\wedge x\wedge y_{i+1}\wedge\cdots\wedge y_d\right).
\end{eqnarray*}
By specifying  $x,y_0,\ldots,y_d$ in
\[\psi_{X,K}\left(\xi(x)y_0\wedge\cdots\wedge y_d-\sum_{i=0}^d\xi(y_i)y_0\wedge\cdots\wedge y_{i-1}\wedge x\wedge y_{i+1}\wedge\cdots\wedge y_d\right)\]
we obtain a linear system $I_{X,K}$ of homogeneous polynomials of degree $\delta$ in $\mathcal E_K^\vee$, which defines a sub-scheme $\widetilde{X}$ of $\mathbb P(\mathcal E_K)$. In fact, an anti-symmetric homomorphism $\sE_K^\vee\rightarrow\sE_K$ acting on an element $\xi$ in $\sE_K^\vee$ can be written as a combination over $K$ of elements of the form $\xi(x)y-\xi(y)x$, where $x$ and $y$ are elements in $\sE_K$.
\begin{defi}\label{definition of wmu(I_X)}
  With all the above notations and operations, let $I_X$ be the largest saturated sub-$\O_K$-module of $\sym^\delta(\mathcal E)$ such that $I_X\otimes_{\O_K}K=I_{X,K}$. We define $\overline{I}_X$ as the Hermitian sub-vector bundle of $\sym^\delta(\overline{\mathcal E})$ equipped with the induced norm by the symmetric norms over $\sym^{\delta}(\overline{\mathcal E})$.
\end{defi}
\begin{rema}\label{X is hypersurface I_X}
  We consider a special case. Let $X$ be a hypersurface of $\mathbb P(\sE_K)$ which is of degree $\delta$ defined by $f\in\sym^\delta(\sE_K^\vee)$. In this case, the module $I_{X,K}$ is generated by the element $f$. We refer the readers to \cite[Examples 2.3 (b), \S3.2.B]{Gelfandal94} for a proof.
\end{rema}
\subsubsection*{A height function induced from Cayley forms}
We consider the Cayley form of $X$ defined in Definition \ref{cayley form 2}. By the argument in Remark \ref{X is hypersurface I_X}, the system $I_{X,K}$ defined in Definition \ref{definition of wmu(I_X)} is generated by the polynomial $f(T_0,\ldots,T_n)$, and in fact $\overline{I}_X$ if a Hermitian line bundle over $\spec \O_K$. Then
\begin{equation}\label{explicit I_X when X is hypersurface}
\wmu(\overline{I}_X)=-\sum_{\p\in\spm\O_K}\frac{[K_v:\Q_v]}{[K:\Q]}\log\|f\|_{\p}-\sum_{v\in M_{K,\infty}}\frac{[K_v:\Q_v]}{[K:\Q]}\log\|f\|_{v,\mathrm{sym}},
\end{equation}
where $\|\ndot\|_{v,\mathrm{sym}}$ is the symmetric norm over $\sym^\delta(\sE_K^\vee)$ for any $v\in M_{K,\infty}$.
\section{Arithmetic Hilbert-Samuel function of hypersurfaces}\label{uniform bound of arithmetic Hilbert}
In this section, we will give both an upper and a lower bounds of the arithmetic Hilbert-Samuel function $\F_D$ defined in Definition \ref{arithmetic hilbert function}. We will first recall a naive lower bound of $\F_D$ considered in \cite{Chen1,Chen2}. In order to get better bounds, we will consider an uniform estimate of the arithmetic Hilbert-Samuel function of projective spaces with respect to the universal bundle, and we will give an estimate for the case of hypersurfaces by comparing some norms.
\subsection{A naive lower bound}
By \cite[Proposition 2.9]{Chen1} and \cite[\S 2.4]{Chen2}, we have the following naive uniform lower bounds of $\F_D$
\begin{equation}\label{trivial estimate of F_D}
  \wmu(\F_D)\geqslant\wmu_{\min}(\E_D)\geqslant-\frac{1}{2}D\log(n+1),
\end{equation}
which is verified uniformly for all integers $D\geqslant1$.
\subsection{Arithmetic Hilbert-Samuel function of projective spaces}

In this paragraph, we will give an estimate of the arithmetic Hilbert-Samuel function of the projective space of dimension $n$ with respect to the universal bundle.
\begin{prop}[\cite{Gaudron06}, Proposition 4.2]\label{arithmetic hilbert for Pn}
  Let $\E_{D,\mathrm{sym}}$ be the same as in Definition \ref{definition of E_D with norm}, and $r(n,D)=\rg(E_D)={n+D\choose D}$, $D\geqslant1$. Then we have
  \begin{equation*}
    \wmu(\E_{D,\mathrm{sym}})=-\frac{1}{2r(n,D)}\sum\limits_{\begin{subarray}{c} i_0+\cdots+i_n=D \\ i_0,\ldots,i_n\geqslant0\end{subarray}}\log\left(\frac{D!}{i_0!\cdots i_n!}\right)+D\wmu(\overline{\mathcal{E}}),
  \end{equation*}
  where $\overline{\sE}$ is a Hermitian vector bundle of rank $n+1$ over $\spec\O_K$.
\end{prop}
\begin{rema}
  With all the notations in Proposition \ref{arithmetic hilbert for Pn}. If $D=0$, we have $\wmu(\E_D)=0$, puisque $\E_D\cong \O_K$. Si $D<0$, le $\O_K$-module $E_D$ est nul.
\end{rema}

Let
\begin{equation}\label{constant C}
C(n,D)=\sum\limits_{\begin{subarray}{c} i_0+\cdots+i_n=D \\ i_0,\ldots,i_n\geqslant0\end{subarray}}\log\left(\frac{i_0!\cdots i_n!}{D!}\right),
\end{equation}
then we have
\[\adeg_n(\E_{D,\mathrm{sym}})=\frac{1}{2}C(n,D)+Dr(n,D)\wmu(\mathcal E)\]
by Proposition \ref{arithmetic hilbert for Pn}. By Proposition \ref{symmetric norm vs John norm, constant}, we obtain
\begin{equation*}
 \wmu(\E_D)=\frac{1}{2r(n,D)}C(n,D)+D\wmu(\overline{\mathcal{E}})-R_0(n,D),
\end{equation*}
where $\E_D$ is equipped with the norms of John induced by the spermium norms when $D\geqslant0$, and the constant $R_0(n,D)$ is defined in Proposition \ref{symmetric norm vs John norm, constant}, which satisfies
\begin{equation*}
  0\leqslant R_0(n,D)\leqslant\log\sqrt{r(n,D)}.
\end{equation*}
See Definition \ref{definition of E_D with norm}.

By the proposition in \cite[Annexe]{Gaudron08}, we have
\begin{eqnarray*}
  C(n,D)&=&-\left(\frac{1}{2}+\cdots+\frac{1}{n+1}\right)\left(D+o(D)\right)r(n,D)\\
&=&-\frac{1}{n!}\left(\frac{1}{2}+\cdots+\frac{1}{n+1}\right)D^{n+1}+o(D^{n+1})
\end{eqnarray*}
 when $n\geqslant2$, but the estimate of remainder is implicit. In fact, we have a finer explicit estimate, which is (in Theorem \ref{estimate of C(n,D)})
\begin{eqnarray}\label{estimation of C}
C(n,D)&=&\frac{1-\mathcal H_{n+1}}{n!}D^{n+1}-\frac{n-2}{2n!}D^n\log D\\
& &+\frac{1}{n!}\Biggr(\left(-\frac{1}{6}n^3-\frac{3}{4}n^2-\frac{13}{12}n+2\right)\mathcal H_n\nonumber\\
& &\:+\frac{1}{4}n^3+\frac{17}{24}n^2+\left(\frac{119}{72}-\frac{1}{2}\log\left(2\pi\right)\right)n-4+\log\left(2\pi\right)\Biggr)D^n\nonumber\\
& &+o(D^n),\nonumber
\end{eqnarray}
where $\mathcal H_n=1+\frac{1}{2}+\cdots+\frac{1}{n}$. In addition, we give both a uniform lower bound and an upper bound of the reminder explicitly. We will give the details of the above estimate in the appendix of this paper.

\subsection{An upper bound and a lower bound of the arithmetic Hilbert-Samuel function of hypersurfaces}
In this part, following Definition \ref{arithmetic hilbert function}, we will give both an upper and a lower bounds of the arithmetic Hilbert-Samuel function of a hypersurface of $\mathbb P(\sE_K)$ with respect to the universal bundle, which are explicit and uniform, where we consider the case of $\overline{\sE}=\left(\O_K^{\oplus(n+1)},\left(\|\ndot\|_v\right)_{v\in M_{K,\infty}}\right)$ equipped with $\ell^2$-norms. This means that for every embedding $v:K\hookrightarrow\C$, the norm $\|\ndot\|_v$ which is with respect to this embedding maps $(x_0,\ldots,x_n)$ to
\begin{equation}\label{l^2-norm}
\sqrt{|v(x_0)|^2+\cdots+|v(x_n)|^n}.
 \end{equation}
In this case, we have $\wmu\left(\overline{\sE}\right)=0$ et $\mathbb P(\sE_K)=\mathbb P^n_K$. The arithmetic Hilbert-Samuel function follows Definition \ref{arithmetic hilbert function}. These estimates are better than the estimate \cite{David_Philippon99} and \cite{Chen1} in the case of projective hypersurfaces.

For this target, first we introduce some numerical lemmas.
\begin{lemm}\label{constant of Jia Jiao}
  Let $i_0,\ldots,i_n,j_n,\ldots,j_n$ be some positive integers, which satisfy $i_0+\cdots+i_n=D$ and $j_0+\cdots+j_n=D'$. Let
  \begin{equation}\label{constant G}
G(D',n)=\left\{
    \begin{array}{ll}
      (n+1)^n, & \hbox{si $D'\leqslant n$;} \\
      \frac{e^{2n+3}D'^{\frac{n}{2}}}{(2\pi)^\frac{n+3}{2}(n+1)^{\frac{n+1}{2}}}, & \hbox{si $D'\geqslant n+1$.}
    \end{array}
  \right.
\end{equation}
  Then we have
\[G(D',n)^{-1}\leqslant\frac{{i_0+j_0\choose i_0}\cdots{i_n+j_n\choose i_n}}{{D+D'\choose D}}\leqslant1.\]
\end{lemm}
\begin{proof}
   The inequality
\[\frac{{i_0+j_0\choose i_0}\cdots{i_n+j_n\choose i_n}}{{D+D'\choose D}}\leqslant1\]
is obtained by the definition directly.

For the other side, if one of the $D$ and $D'$ is larger than or equal to $n$, then we can suppose that $i_0\geqslant j_1\geqslant\cdots\geqslant j_n\geqslant0$, and $D'\leqslant n$. In this case, if $D\leqslant n$, we obtain
\begin{equation*}
  \frac{{i_0+j_0\choose i_0}\cdots{i_n+j_n\choose i_n}}{{D+D'\choose D}}\geqslant\frac{1}{{D+D'\choose D}}\geqslant\frac{1}{{2n\choose n}}\geqslant\frac{1}{(n+1)^n};
\end{equation*}
if $D\geqslant n+1$, then we have
\begin{equation*}
  \frac{{i_0+j_0\choose i_0}\cdots{i_n+j_n\choose i_n}}{{D+D'\choose D}}\geqslant\frac{(\frac{D}{n+1}+1)^{D'}}{(D+1)^{D'}}\geqslant\left(\frac{2}{n+2}\right)^{D'}\geqslant\frac{1}{(n+1)^n}.
\end{equation*}

If $D$ and $D'$ are both larger than or equal to $n+1$, by the Stirling formula
\begin{equation*}
  \sqrt{2\pi}m^{m+\frac{1}{2}}e^{-m}\leqslant m!\leqslant em^{m+\frac{1}{2}}e^{-m},
\end{equation*}
we have the inequality
\[\frac{{i_0+j_0\choose i_0}\cdots{i_n+j_n\choose i_n}}{{D+D'\choose D}}
\geqslant\frac{(2\pi)^\frac{n+3}{2}}{e^{2n+3}}\cdot\frac{(i_0+j_0)^{i_0+j_0+\frac{1}{2}}\cdots(i_n+j_n)^{i_n+j_n+\frac{1}{2}}D^{D+\frac{1}{2}}D'^{D'+\frac{1}{2}}}{i_0^{i_0+\frac{1}{2}}j_0^{j_0+\frac{1}{2}}\cdots i_n^{i_n+\frac{1}{2}}j_n^{j_n+\frac{1}{2}}(D+D')^{D+D'+\frac{1}{2}}}.\]

Let
\[F(i_0,\ldots,i_n,j_0,\ldots,j_n)=\frac{(i_0+j_0)^{i_0+j_0+\frac{1}{2}}\cdots(i_n+j_n)^{i_n+j_n+\frac{1}{2}}}{i_0^{i_0+\frac{1}{2}}j_0^{j_0+\frac{1}{2}}\cdots i_n^{i_n+\frac{1}{2}}j_n^{j_n+\frac{1}{2}}},\]
where $i_0+\cdots+i_n=D$, $j_0+\cdots+j_n=D'$, and $i_0,\ldots,i_n,j_0,\ldots,j_n\geqslant0$. If we consider $F(i_0,\ldots,i_n,j_0,\ldots,j_n)$ as a function of variables $(i_0,\ldots,i_n,j_0,\ldots,j_n)\in\mathbb ]1,+\infty[^{2n+2}$, we can confirm that it take the minimal value when $i_0=\cdots=i_n=\frac{D}{n+1}$, $j_0=\cdots=j_n=\frac{D'}{n+1}$. Then we have
\begin{eqnarray*}
  \frac{F(i_0,\ldots,i_n,j_0,\ldots,j_n)}{\frac{(D+D')^{D+D'+\frac{1}{2}}}{D^{D+\frac{1}{2}}D'^{D'+\frac{1}{2}}}}&\geqslant&\frac{\left(\frac{D+D'}{n+1}\right)^{D+D'+\frac{n+1}{2}}D^{D+\frac{1}{2}}D'^{D'+\frac{1}{2}}}{(D+D')^{D+D'+\frac{1}{2}}\left(\frac{D}{n+1}\right)^{D+\frac{n+1}{2}}\left(\frac{D'}{n+1}\right)^{D+\frac{n+1}{2}}}\\
&=&\frac{(n+1)^{\frac{n+1}{2}}(D+D')^{\frac{n}{2}}}{D^{\frac{n}{2}}D'^{\frac{n}{2}}}\\
&\geqslant&\frac{(n+1)^{\frac{n+1}{2}}}{D'^{\frac{n}{2}}},
\end{eqnarray*}
which terminates the proof.
\end{proof}
\begin{lemm}\label{product}
Let $K$ be a number field, and $\O_K$ be the ring of integers of $K$. For the Hermitian vector bundle $\overline{\sE}=\left(\O_K^{\oplus(n+1)}, (\|\ndot\|_v)_{v\in M_K}\right)$ equipped with the $\ell^2$-norms defined in \eqref{l^2-norm}, let $f\in H^0\left(\mathbb{P}(\mathcal{E}_K),\O_{\mathbb P(\mathcal E_K)}(D)\right)$ and $g\in H^0\left(\mathbb{P}(\mathcal{E}_K),\O_{\mathbb P(\mathcal E_K)}(D')\right)$. Then for every place $v\in M_{K,\infty}$, we have
  \begin{eqnarray*}
    & &\log\|f\|_{v,\mathrm{sym}}+\log\|g\|_{v,\mathrm{sym}}-\frac{1}{2}\log G(D',n)\\
&\leqslant&\log\|f\cdot g\|_{v,\mathrm{sym}}\leqslant\log\|f\|_{v,\mathrm{sym}}+\log\|g\|_{v,\mathrm{sym}},
  \end{eqnarray*}
 where the constant $G(D',n)$ is defined in the equality \eqref{constant G}. 

If $v\in M_{K,f}$, we have
\[\|f\cdot g\|_v=\|f\|_v\cdot\|g\|_v.\]
\end{lemm}
\begin{proof}
  First, we consider the case of $v\in M_{K,\infty}$. By the properties of Banach algebra, we have
  \[\log\|f\cdot g\|_{v,\mathrm{sym}}\leqslant\log\|f\|_{v,\mathrm{sym}}+\log\|g\|_{v,\mathrm{sym}}.\]

  For the other side, we suppose that $B_t=\{e_0^{i_0}e_1^{i_1}\cdots e_n^{i_n}|i_0+i_1+\cdots+i_n=t,\; i_0,i_1, \ldots, i_n\in\mathbb N\}$ is a canonical orthogonal basis of $E_t$ equipped with the symmetric norm with respect to the orthonormal basis $e_0,e_1,\ldots,e_n$ of $\sE\otimes_{\O_K,v}\C$, where $t\in\mathbb N_+$. In this proof, we consider the case where $t=D$ and $t=D'$. First we suppose that $f$ and $g$ are two elements in the canonical basis $H^0\left(\mathbb{P}(\mathcal{E}_K),\O_{\mathbb P(\mathcal E_K)}(D)\right)$ and $H^0\left(\mathbb{P}(\mathcal{E}_K),\O_{\mathbb P(\mathcal E_K)}(D')\right)$ respectively defined above. If $f=e_0^{i_0}e_1^{i_1}\cdots e_n^{i_n}$ and $g=e_0^{j_0}e_1^{j_1}\cdots e_n^{j_n}$, then we have
  \begin{equation*}
    f\cdot g=e_0^{i_0+j_0}e_1^{i_1+j_1}\cdots e_n^{i_n+j_n}\in E_{D+D'}.
  \end{equation*}
By \cite[Chap. V, \S3.3]{Bourbaki81}, we obtain
\[\|f\cdot g\|_{v,\mathrm{sym}}=\sqrt{\frac{(i_0+j_0)!\cdots(i_n+j_n)!}{(D+D')!}}\]
and
\[\|f\|_{v,\mathrm{sym}}\cdot\|g\|_{v,\mathrm{sym}}=\sqrt{\frac{i_0!\cdots i_n!}{D!}}\cdot\sqrt{\frac{j_0!\cdots j_n!}{D'!}}.\]
By Lemma \ref{constant of Jia Jiao}, we obtain
  \begin{equation*}
    \frac{\|f\cdot g\|_{v,\mathrm{sym}}}{\|f\|_{v,\mathrm{sym}}\cdot\|g\|_{v,\mathrm{sym}}}\geqslant\sqrt{G(D',n)^{-1}},
  \end{equation*}
which means
\begin{equation*}
  \log\|f\cdot g\|_{v,\mathrm{sym}}\geqslant\log\|f\|_{v,\mathrm{sym}}+\log\|g\|_{v,\mathrm{sym}}-\frac{1}{2}\log G(D',n).
\end{equation*}

For the general case, we consider the set $\{a\cdot b|\;a\in B_D, b\in B_{D'}\}=B_{D+D'}$, which is an orthogonal basis of $E_{D+D'}$ equipped with the symmetric norm.
We denote by $B_D=\{x_i\}_{i\in I}$ et $B_{D'}=\{y_j\}_{i\in J}$ for simplicity, where $I=\{(i_0,\ldots,i_n)|\;i_0+\cdots+i_n=D\}$ and $J=\{(i_0,\ldots,i_n)|\;i_0+\cdots+i_n=D'\}$ are the index sets. Then the set $\{x_i y_j\}_{i\in I, j\in J}=B_{D+D'}$ form an orthogonal basis of $E_{D+D'}$. If $f=ax_i$ and $g=by_j$ for some $i\in I$ and $j\in J$, where $a,b\in K$, we obtain
\begin{eqnarray*}
  & &\log\|f\cdot g\|_{v,\mathrm{sym}}\\
  &=&\frac{1}{2}\log\left(|a|_v^2|b|_v^2\langle x_iy_j, x_iy_j\rangle_v\right)\\
  &\geqslant&\frac{1}{2}\Big(\log(|a|_v^2)+\log(|b|_v^2)+\log\langle x_i,x_i\rangle_v+\log\langle y_j,y_j\rangle_v-\log G(D',n)\Big)\\
  &=&\log\|f\|_{v,\mathrm{sym}}+\log\|g\|_{v,\mathrm{sym}}-\frac{1}{2}\log G(D',n).
\end{eqnarray*}

If $f=\sum\limits_{i\in I} a_ix_i$ and $g=\sum\limits_{j\in J} b_jy_j$, where $a_i,b_j\in K$, $x_i\in B_{D}$, and $y_i\in B_{D'}$ are chosen as above, we obtain
\begin{eqnarray*}
  \log\|f\cdot g\|_{v,\mathrm{sym}}&=&\frac{1}{2}\log\left(\sum_{i\in I,j\in J}\langle a_ib_jx_iy_j, a_ib_jx_iy_j\rangle_v\right)\\
  &\geqslant&\frac{1}{2}\log\left(\sum\limits_{i\in I,j\in J}\frac{\langle a_ix_i,a_ix_i\rangle_v\cdot\langle b_jy_j,b_jy_j\rangle_v}{G(D',n)}\right)\\
  &=&\frac{1}{2}\log\left(\sum\limits_{i\in I}\langle a_ix_i,a_ix_i\rangle_v\cdot\sum\limits_{j\in J}\langle b_jy_j,b_jy_j\rangle_v\right)-\frac{1}{2}\log G(D',n)\\
  &=&\log\|f\|_{v,\mathrm{sym}}+\log\|g\|_{v,\mathrm{sym}}-\frac{1}{2}\log G(D',n),
\end{eqnarray*}
and we prove the assertion for $v\in M_{K,\infty}$.

For the case of $v\in M_{K,f}$, it is showed by the definition of the discrete valuation directly.
\end{proof}

Let $\overline{\sE}=\overline{\O}_K^{\oplus(n+1)}$ be the Hermitian vector bundle equipped with $\ell^2$-norms defined in \eqref{l^2-norm}, and $X$ be the hypersurface of $\mathbb P^n_K$ defined by the homogeneous polynomial $f(T_0,\ldots,T_n)$ of degree $\delta$. Let $s$ be the non-zero global section in $H^0\left(\mathbb P^n_{\O_K},\O_{\mathbb P^n_{\O_K}}\left(\delta\right)\right)$ which defines the Zariski closure of $X$ in $\mathbb P^n_{\O_K}$. Then we have the following short exact sequence of $\O_K$-modules:
\begin{equation}\label{exact sequence of F_D}
\begin{CD}
0@>>>E_{D-\delta}@>\cdot s>>E_D@>>>F_D@>>>0,
\end{CD}\end{equation}
where $E_{D-\delta}$ and $E_D$ are defined in \eqref{definition of E_D}, and $F_D$ is the saturated image of the map \eqref{evaluation map}. The third arrow in \eqref{exact sequence of F_D} is the canonical quotient morphism.

In order to estimate the arithmetic Hilbert-Samuel function of the hypersurface $X$, we have the following result.
\begin{theo}\label{upper and lower bound of arithmetic Hilbert}
 Let $X$ be a hypersurface of degree $\delta$ of $\mathbb P^n_K$, and $D$ be a positive integer. Let the constant $R_0(n,D)$ be as in Proposition \ref{symmetric norm vs John norm, constant}, the constant $G(\delta,n)$ be as in \eqref{constant G}, the constant $C(n,D)$ be as in \eqref{constant C}, the vector bundle $\F_D$ over $\spec\O_K$ be as in Definition \ref{arithmetic hilbert function}, $r_1(n,D)=\rg(F_D)$, and $\wmu(\overline I_X)$ be as in Definition \ref{definition of wmu(I_X)}. Then the inequalities
\begin{eqnarray*}
\wmu(\F_D)&\geqslant&\frac{1}{2r_1(n,D)}\Big(C(n,D)-C(n,D-\delta)-2r(n,D-\delta)\wmu(\overline{I}_X)\\
& &-r(n,D-\delta)\log G(\delta,n)\Big)-R_0(n,D),
\end{eqnarray*}
and
\begin{eqnarray*}
  \wmu(\F_D)&\leqslant&\frac{1}{2r_1(n,D)}\Big(C(n,D)-C(n,D-\delta)-2r(n,D-\delta)\wmu(\overline{I}_X)\Big)\\
  & &-R_0(n,D).
\end{eqnarray*}
are uniformly verified for any $D\geqslant\delta$. If $D\leqslant\delta-1$, we have
\[\wmu(\F_D)=\frac{1}{2r(n,D)}C(n,D)-R_0(n,D).\]
\end{theo}
\begin{proof}
  First, we consider the case where $D\geqslant\delta$. We suppose that $\adeg'_n(\E_{D-\delta})$ is the normalized Arakelov degree of $E_{D-\delta}$ equipped with the norms as a Hermitian sub-bundle of the Hermitian vector bundle $\E_{D}$ defined via the exact sequence \eqref{exact sequence of F_D}. Then we obtain
  \begin{equation}\label{additive of F_D}
    \adeg_n(\F_D)=\adeg_n(\E_D)-\adeg'_n(\E_{D-\delta})
  \end{equation}
by the equality \eqref{Po}.

We need to compare the norms over $E_{D-\delta}$ as a Hermitian sub-vector bundle of $\E_{D}$ with the norms of John over $E_{D-\delta}$ defined above. Let $\{e_1,\ldots,e_N\}$ be a orthogonal basis of $E_{D-\delta}$ under the symmetric norms, where we note $N=r(n,D-\delta)$ for simplicity. By definition, we obtain
\begin{equation*}
  \adeg'_n(\E_{D-\delta})=-\sum\limits_{v\in M_K}\frac{[K_v:\Q_v]}{[K:\Q]}\log\|(fe_1)\wedge(fe_2)\wedge\cdots\wedge(fe_N)\|_{v,J},
\end{equation*}
where $\|\ndot\|_{v,J}$ is the norm of John over $\E_D$ induced by the supremum norm at the place $v\in M_{K,\infty}$.

By the equality \eqref{symmetric norm vs John norm} and Proposition \ref{symmetric norm vs John norm, constant}, we obtain
\begin{eqnarray}\label{symetric norm of F_D}
\adeg'_n(\E_{D-\delta})&=&-\sum\limits_{v\in M_K}\frac{[K_v:\Q_v]}{[K:\Q]}\log\|(fe_1)\wedge(fe_2)\wedge\cdots\wedge(fe_N)\|_{v,\mathrm{sym}}\\
& &-r(n,D-\delta)R_0(n,D)\nonumber
\end{eqnarray}
and
\begin{equation}\label{symetric norm of E_D}
  \adeg_n(\E_D)=\frac{1}{2}C(n,D)-r(n,D)R_0(n,D),
\end{equation}
where $R_0(n,D)$ is defined in Proposition \ref{symmetric norm vs John norm, constant}.

We will estimate the term $\|(fe_1)\wedge(fe_2)\wedge\cdots\wedge(fe_N)\|_{v,\mathrm{sym}}$, where $v\in M_{K,\infty}$. Since $\{fe_1,fe_2,\ldots,fe_N\}$ is a base of $E_D$. Let $\langle\ndot,\ndot\rangle_v$ be the scalar product induced by the norm $\|\ndot\|_{v,\mathrm{sym}}$ over $E_D$ for any $v\in M_{K,\infty}$. Then by Lemma \ref{product}, we obtain the inequality
\begin{eqnarray}\label{upper bound of infinite place of F_D}
& &\log\|(fe_1)\wedge(fe_2)\wedge\cdots\wedge(fe_N)\|_{v,\mathrm{sym}}\\
&\leqslant&\log(\|f\|_{v,\mathrm{sym}}^N\cdot\|e_1\wedge\cdots\wedge e_N\|_{v,\mathrm{sym}})\nonumber\\
&=&N\log\|f\|_{v,\mathrm{sym}}+\log\|e_1\wedge\cdots\wedge e_N\|_{v,\mathrm{sym}},\nonumber
\end{eqnarray}
and the inequality
\begin{eqnarray}\label{lower bound of infinite place of F_D}
& &\log\|(fe_1)\wedge(fe_2)\wedge\cdots\wedge(fe_N)\|_{v,\mathrm{sym}}\\
&=&\frac{1}{2}\log\det\left(\langle fe_i,fe_j\rangle_v\right)_{1\leqslant i,j\leqslant N}\nonumber\\
&\geqslant&\frac{1}{2}\log\left(\det\left(\langle e_i,e_j\rangle_v\right)_{1\leqslant i,j\leqslant N}\cdot\left(\frac{\langle f,f\rangle_v}{G(\delta,n)}\right)^N\right)\nonumber\\
&=&\frac{1}{2}\log\det\left(\langle e_i,e_j\rangle_v\right)_{1\leqslant i,j\leqslant N}+\frac{1}{2}\log\langle f,f\rangle_v^N-\frac{1}{2}\log G(\delta,n)^{N}\nonumber\\
&=&\log\|e_1\wedge\cdots\wedge e_N\|_{v,\mathrm{sym}} +N\log\|f\|_{v,\mathrm{sym}}-\frac{N}{2}\log G(\delta,n).\nonumber
\end{eqnarray}

If $v\in M_{K,f}$, we have
\begin{eqnarray}\label{finite place of F_D}
\log\|(fe_1)\wedge(fe_2)\wedge\cdots\wedge(fe_N)\|_{v}&=&\log(\|f\|_{v}^N\cdot\|e_1\wedge\cdots\wedge e_N\|_{v})\\
&=&N\log\|f\|_{v}+\log\|e_1\wedge\cdots\wedge e_N\|_{v}\nonumber
\end{eqnarray}
by the definition of discrete valuation. So by \eqref{symetric norm of F_D}, \eqref{lower bound of infinite place of F_D}, and \eqref{finite place of F_D}, we have
\begin{eqnarray*}
  \adeg'_n(\E_{D-\delta})&\leqslant&\frac{1}{2} C(n,D-\delta)+r(n,D-\delta)\wmu(\overline{I}_X)+\frac{r(n,D-\delta)}{2}\log G(\delta,n)\\
& &-r(n,D-\delta)R_0(n,D),
\end{eqnarray*}
where the slope $\wmu(\overline I_X)$ is obtained from \eqref{explicit I_X when X is hypersurface} for the case of hypersurfaces. Then by \eqref{additive of F_D} and \eqref{symetric norm of E_D}, we obtain
\begin{eqnarray*}
\wmu(\F_D)&\geqslant&\frac{1}{r_1(n,D)}\Biggr(\frac{1}{2}C(n,D)-R_0(n,D)\Big(r(n,D)-r(n,D-\delta)\Big)\\
& &-\frac{1}{2}C(n,D-\delta)-r(n,D-\delta)\wmu(\overline{I}_X)-\frac{r(n,D-\delta)}{2}\log G(\delta,n)\Biggr)\\
&=&\frac{1}{2r_1(n,D)}\Big(C(n,D)-C(n,D-\delta)-2r(n,D-\delta)\wmu(\overline{I}_X)\\
& &-r(n,D-\delta)\log G(\delta,n)\Big)-R_0(n,D).
\end{eqnarray*}

By the similar argument, we obtain
\begin{eqnarray*}
  \wmu(\F_D)&\leqslant&\frac{1}{2r_1(n,D)}\Big(C(n,D)-C(n,D-\delta)-2r(n,D-\delta)\wmu(\overline{I}_X)\Big)\\
  & &-R_0(n,D)
\end{eqnarray*}
by \eqref{additive of F_D}, \eqref{symetric norm of F_D}, \eqref{symetric norm of E_D}, \eqref{upper bound of infinite place of F_D}, and \eqref{finite place of F_D}.

If $D\leqslant\delta-1$, then $\E_D\cong \F_D$ as Hermitian vector bundles over $\spec\O_K$, so we accomplish the proof.
\end{proof}

\subsection{Numerical estimate of the arithmetic Hilbert-Samuel function}
Let $X$ be a hypersurface of $\mathbb P^n_K$. In Theorem \ref{upper and lower bound of arithmetic Hilbert}, we have already give an upper bound and a lower bound of its arithmetic Hilbert-Samuel function (see Definition \ref{arithmetic hilbert function}) of $X$, which are both uniform and explicit. In this part, we will give a numerical estimate of the lower bound of the arithmetic Hilbert-Samuel function of $X$, which will be useful in a counting rational points problem below.

\begin{prop}\label{numerical result of arithmetic hilbert of hypersurface}
  Let $\wmu(\overline{I}_X)$ and $\F_D$ be same as in Theorem \ref{upper and lower bound of arithmetic Hilbert}. Let the constants $A_4(n,D)$ and $A_4'(n,D)$ be same as in Theorem \ref{estimate of C(n,D)}, the constant $G(n,\delta)$ defined in the equality \eqref{constant G}, the constant
  \[r(n,D)={n+D\choose n},\]
  and the constant
  \[\mathcal H_n=1+\frac{1}{2}+\cdots+\frac{1}{n}.\]
  We suppose
  \begin{eqnarray*}
  & &B_0(n,\delta)\\
  &=&-\frac{\log G(\delta,n)}{\frac{r(n,\delta+1)}{n+1}-1}-\frac{\log(n+1)}{2}+\frac{1-\mathcal H_{n+1}}{2^{n-1}n!}(n+1) -\frac{n-2}{2^n(n-1)!}\\
& &+\frac{1}{2^{n-1}(n-1)!(\delta+1)}\Biggr(\left(-\frac{1}{6}n^3-\frac{3}{4}n^2-\frac{13}{12}n+2\right)\mathcal H_n\\
& &\:+\frac{1}{4}n^3+\frac{17}{24}n^2+\left(\frac{119}{72}-\frac{1}{2}\log\left(2\pi\right)\right)n-4+\log\left(2\pi\right)\Biggr)\\
& &+\inf_{D\geqslant\delta}\frac{A_4(n,D)-A'_4(n,D-\delta)}{2^{n-1}\delta D^{n}}.
\end{eqnarray*}
  Then the inequalities
  \[\frac{\wmu(\F_D)}{D}\geqslant-\frac{\wmu(\overline I_X)}{n\delta}+B_0(n,\delta)\]
  is uniformly verified for any $D\geqslant\delta$.
\end{prop}

\begin{proof}
By Theorem \ref{upper and lower bound of arithmetic Hilbert}, we have
\begin{eqnarray*}
\wmu(\F_D)&\geqslant&\frac{1}{2r_1(n,D)}\Big(C(n,D)-C(n,D-\delta)-2r(n,D-\delta)\wmu(\overline{I}_X)\\
& &-r(n,D-\delta)\log G(\delta,n)\Big)-R_0(n,D),
\end{eqnarray*}
where the inequality
\[0\leqslant R_0(n,D)\leqslant\log\sqrt{r(n,D)}\leqslant\frac{D}{2}\log(n+1).\]
is verified . Then we have
\[\frac{R_0(n,D)}{D}\leqslant\frac{\log(n+1)}{2}\]
when $D\geqslant\delta$.

 First, we have the inequality $\wmu(\overline I_X)<0$ by definition directly. In addition, since
\[\frac{r(n,\delta+1)}{n+1}\leqslant\frac{r(n,D)}{r(n,D-\delta)}\leqslant\frac{D^n}{(D-\delta)^n}\leqslant 1-\frac{n\delta}{D},\]
then we have
  \[\frac{1}{\frac{r(n,\delta+1)}{n+1}-1}\geqslant\frac{r(n,D-\delta)}{D\left(r(n,D)-r(n,D-\delta)\right)}\geqslant\frac{1}{n\delta}.\]

We have
\[r_1(n,D)=r(n,D)-r(n,D-\delta)={n+D\choose D}-{n+D-\delta\choose D-\delta}\]
from the short exact sequence \eqref{exact sequence of F_D} and the definition of $r(n,D)$. Then we have
 \begin{eqnarray*}
   r_1(n,D)&=&\sum_{j=0}^{\delta-1}{D-\delta+n+j\choose n-1}\\
   &\leqslant&\delta{D+n-1\choose n-1}\\
   &\leqslant&\delta(D+1)^{n-1}\\
   &\leqslant&2^{n-1}\delta D^{n-1}.
 \end{eqnarray*}
  By Theorem \ref{estimate of C(n,D)}, we obtain
\begin{eqnarray*}
  & &C(n,D)-C(n,D-\delta)\\
  &\geqslant&\frac{1-\mathcal H_{n+1}}{n!}(n+1)\delta D^{n}-\frac{n-2}{2n!}n\delta D^{n-1}\log D\\
& &+\frac{1}{n!}\Biggr(\left(-\frac{1}{6}n^3-\frac{3}{4}n^2-\frac{13}{12}n+2\right)\mathcal H_n\\
& &\:+\frac{1}{4}n^3+\frac{17}{24}n^2+\left(\frac{119}{72}-\frac{1}{2}\log\left(2\pi\right)\right)n-4+\log\left(2\pi\right)\Biggr)n\delta D^{n-1}\\
& &+A_4(n,D)-A'_4(n,D-\delta).
\end{eqnarray*}
Then we have
\begin{eqnarray*}
  & &\frac{C(n,D)-C(n,D-\delta)}{2Dr_1(n,D)}\\
  &\geqslant&\frac{1-\mathcal H_{n+1}}{2^{n-1}n!}(n+1) -\frac{n-2}{2^n(n-1)!}\frac{\log D}{D}\\
& &+\frac{1}{2^{n-1}(n-1)!D}\Biggr(\left(-\frac{1}{6}n^3-\frac{3}{4}n^2-\frac{13}{12}n+2\right)\mathcal H_n\\
& &\:+\frac{1}{4}n^3+\frac{17}{24}n^2+\left(\frac{119}{72}-\frac{1}{2}\log\left(2\pi\right)\right)n-4+\log\left(2\pi\right)\Biggr)\\
& &+\frac{A_4(n,D)-A'_4(n,D-\delta)}{2^{n-1}\delta D^{n}}.
\end{eqnarray*}
By the construction of $A_4(n,D)$ and $A'_4(n,D)$ in Theorem \ref{estimate of C(n,D)}, the term
\[\frac{A_4(n,D)-A'_4(n,D-\delta)}{2^{n-1}\delta D^{n}}\]
is uniformly bounded considered as a function of the variable $D$, where $D\geqslant\delta$.

Since $X$ is a hypersurface of degree $\delta$, the constant
\begin{eqnarray*}
  & &B_0(n,\delta)\\
  &=&-\frac{\log G(\delta,n)}{\frac{r(n,\delta+1)}{n+1}-1}-\frac{\log(n+1)}{2}+\frac{1-\mathcal H_{n+1}}{2^{n-1}n!}(n+1) -\frac{n-2}{2^n(n-1)!}\\
& &+\frac{1}{2^{n-1}(n-1)!(\delta+1)}\Biggr(\left(-\frac{1}{6}n^3-\frac{3}{4}n^2-\frac{13}{12}n+2\right)\mathcal H_n\\
& &\:+\frac{1}{4}n^3+\frac{17}{24}n^2+\left(\frac{119}{72}-\frac{1}{2}\log\left(2\pi\right)\right)n-4+\log\left(2\pi\right)\Biggr)\\
& &+\inf_{D\geqslant\delta}\frac{A_4(n,D)-A'_4(n,D-\delta)}{2^{n-1}\delta D^{n}}
\end{eqnarray*}
satisfies the inequality in the assertion.
\end{proof}
\begin{rema}
  With all the notations and conditions in Proposition \ref{numerical result of arithmetic hilbert of hypersurface}, by this proposition and the inequality \eqref{trivial estimate of F_D}, there exists a positive constant $C(X)$ depending on the hypersurface $X$, such that
  \[\frac{\wmu(\F_D)}{D}\geqslant-C(X)\]
  for all $D\in\mathbb N\smallsetminus\{0\}$ uniformly. The case of $D\in\{1,2,\ldots,\delta\}$ is obtained by the isomorphism $\E_D\cong\F_D$.

   Let $\mathscr X$ be a projective scheme, $\mathscr L$ be a Hermitian ample line bundle, and $\G_D=H^0(\mathscr X,\mathscr L|_{\mathscr X}^{\otimes D})$ be a Hermitian vector bundle equipped with the induced norms. By \cite[Lemma 4.8]{Bost2001}, there exists a constant $c_1>0$ which only depends on $\mathscr X$ and $\mathscr L$ such that for any $D\in \mathbb N\smallsetminus\{0\}$, we have
  \[\wmu(\G_D)\geqslant-c_1D.\]
  Then the result of Proposition \ref{numerical result of arithmetic hilbert of hypersurface} can be considered as an example of \cite[Lemma 4.8]{Bost2001} when $\mathscr X$ is a hypersurface and $\mathscr L$ is the universal bundle, for we have $F_D\cong H^0(X,\O_X(D))$ when $X$ is a projective hypersurface.
\end{rema}
\begin{rema}
  By \cite[Proposition 3.6]{Chen1}, we can compare $\wmu(\overline{I}_X)$ and the height $h_{\overline{\mathscr L}}(\mathscr X)$ of $X$ defined by the arithmetic intersection theory. Then Theorem \ref{upper and lower bound of arithmetic Hilbert} covers the estimate of lower bound in \cite{David_Philippon99} and \cite[Theorem 4.8]{Chen1}, for the constants in the above estimate are better.
\end{rema}
\section{An application: the density of rational points with small heights}
We have given an uniform explicit estimate of the arithmetic Hilbert-Samuel function of a projective hypersurface in \S\ref{uniform bound of arithmetic Hilbert}. As an application, we first suppose that $X$ is an integral hypersurface of degree $\delta$ in $\mathbb P^n_K$, and we will construct a hypersurface of degree at most $\delta$ which covers all the rational points of small heights of $X$ but do not contain the generic point of $X$. This kind of results is useful in the application of the determinant method to in counting rational points problems, see \cite[Theorem 4]{Heath-Brown}, \cite[Lemma 3]{Browning_Heath05}, \cite[Lemma 6.3]{Salberger07}, \cite[Lemma 1.7]{Salberger_preprint2013}, and the remark under the statement of \cite[Theomrem 1.3]{Walsh_2015}, for example.
\subsection{Heights of rational points}
Let $K$ be a number field, and $\O_K$ be its ring of integers. In order to describe the arithmetic complexity of the closed points in $\mathbb P^n_K$, we introduce the following height function.
\subsubsection*{Naive height function}
First, we recall the following common definition of height function (cf. \cite[\S B.2]{Hindry}).
\begin{defi}\label{weil height}
Let $\xi\in\mathbb{P}^n_K(K)$. We write a $K$-rational homogeneous coordinate of $\xi$ as $[x_0:\cdots:x_n]$. We define the \textit{absolute logarithmic height} of the point $\xi$ as
\[h(\xi)=\frac{1}{[K:\Q]}\sum_{v\in M_{K}}\log\left(\max_{0\leqslant i\leqslant n}\{|x_i|_v^{[K_v:\Q_v]}\}\right),\]
which is independent of the choice of the projective coordinate by the product formula (cf. \cite[Chap. III, Proposition 1.3]{Neukirch}).
\end{defi}
We can prove that $h(\xi)$ is independent of the choice of the base field $K$ (cf. \cite[Lemma B.2.1]{Hindry}).

In addition, we define the \textit{relative multiplicative height} of the point $\xi$ to be
\[H_{K}(\xi) = \exp\left([K:\Q]h(\xi)\right).\]

When considering the closed points of a sub-scheme $X$ of $\mathbb P^n_K$ with the immersion $\phi:X\hookrightarrow\mathbb P^n_K$, we define the height of $\xi\in X(\overline K)$ to be
\[h(\xi):=h(\phi(\xi)).\]
We shall use this notation when there is no confusion of the immersion morphism $\phi$.

By the Northcott's property (cf. \cite[Theorem B.2.3]{Hindry}), the cardinality $\#\{\xi\in X(K)|\;H_K(\xi)\leqslant B\})$ is finite for every $B\geqslant1$.
\subsubsection*{Reformulated by Arakelov geometry}
The definition of height of a rational points can be defined by the language of Arakelov geometry by the following method.
 \begin{defi}\label{arakelov height}
   Let $\overline{\sE}$ be a Hermitian vector bundle over $\spec\O_K$ of dimension $n+1$, $\pi:\mathbb P(\sE)\rightarrow \spec\O_K$ be the structural morphism, and $\xi$ be a rational point of $\mathbb P(\sE_K)$. The point $\xi$ extends in a unique way to a section $\mathcal P_\xi$ of $\pi$. Let $\overline{\mathcal L}$ be a Hermitian line bundle of $\mathbb P(\overline{\sE})$. The \textit{Arakelov height} of point $\xi$ with respect to $\overline{\mathcal L}$ is defined to be $\adeg_n(\mathcal P_\xi^*\overline{\mathcal L})$, denoted by $h_{\overline{\mathcal L}}(\xi)$.
   \end{defi}
 If we take $\overline{\mathcal L}=\overline{\O(1)}$ and $\overline{\sE}=\left(\O_K^{\oplus(n+1)},\left(\|\ndot\|_v\right)_{v\in M_{K,\infty}}\right)$ equipped with the $\ell^2$-norm defined in \eqref{l^2-norm}, and let $[x_0:\cdots:x_n]$ be a $K$-rational projective coordinate of $\xi$, then we have (cf. \cite[(3.1.6)]{BGS94} or \cite[Proposition 9.10]{Moriwaki-book})
\begin{eqnarray}
  h_{\overline{\O(1)}}(\xi)&=&\sum\limits_{\p\in \spm \O_K}\frac{[K_\p:\Q_\p]}{[K:\Q]}\log \left(\max\limits_{1\leqslant i\leqslant n}|x_i|_\p\right)\\
  & &\;+\frac{1}{2}\sum\limits_{\sigma\in M_{K,\infty}}\frac{[K_\sigma:\Q_\sigma]}{[K:\Q]}\log\left(\sum\limits_{j=0}^n|x_j|_\sigma^2\right),\nonumber
\end{eqnarray}
which is independent of the choice of the projective coordinate by the product formula. By definition, we have
\begin{equation}\label{difference between classical height and arakelov height}
  h(\xi)\leqslant h_{\overline{\O(1)}}(\xi)\leqslant h(\xi)+\frac{1}{2}\log(n+1),
\end{equation}
where the height $h(\xi)$ is defined in Definition \ref{weil height}. So Arakelov height defined in Definition \ref{arakelov height} also satisfies the Northcott property.

In order to use the method of Arakelov geometry, the Arakelov height defined in Definition \ref{arakelov height} is useful in this approach.

Let $B \geqslant1$, and $X$ be the subscheme of $\mathbb P^n_K$ defined. Suppose
\[H_{\overline{\O(1)},K}=\exp\left([K:\Q]h_{\overline{\O(1)}}(\xi)\right).\]
 We denote by
\begin{equation}\label{S(X;B)}
S(X;B)=\left\{\xi\in X(K)|\;H_{\overline{\O(1)},K}(\xi)\leqslant B\right\}.
\end{equation}
In addition, we denote by
\begin{equation}\label{N(X;B)}
N(X;B)=\#S(X;B),
\end{equation}
which is also finite for every fixed $B\geqslant1$ from \eqref{difference between classical height and arakelov height} and the Northcott property introduced above.
\subsection{A comparison of heights of a hypersurface}
In the problem of counting rational points with bounded heights, the following definition of height of a hypersurface is usually useful, see \cite{Heath-Brown}, \cite[Notation 1.6]{Salberger_preprint2013} or \cite[\S1]{Walsh_2015} for instance.
\begin{defi}[Classical height]\label{classical height of hypersurface}
Let
  \[f(T_0,T_1,\ldots,T_n)=\sum_{\begin{subarray}{x}(i_0,\ldots,i_n)\in\mathbb N^{n+1}\\ i_0+\cdots+i_n=\delta\end{subarray}}a_{i_0,i_1,\ldots,i_n}T_0^{i_0}T_1^{i_1}\cdots T_n^{i_n}\]
  be a non-zero homogeneous polynomial with coefficients in $K$.
 The \textit{classical height} $h(f)$ of the polynomial $f$ is defined below:
\begin{equation*}
  h(f)=\sum_{v\in M_K}\frac{[K_v:\Q_v]}{[K:\Q]}\log\max\limits_{\begin{subarray}{x}(i_0,\ldots,i_n)\in\mathbb N^{n+1}\\ i_0+\cdots+i_n=\delta\end{subarray}}\{|a_{i_0,\ldots, i_n}|_v\}.
\end{equation*}
In addition, if $X$ is the hypersurface in $\mathbb P^n_K$ defined by $f$, we define $h(X)=h(f)$ as the \textit{classical height} of the hypersurface $X$. We denote by $H_K(X)=\exp\left([K:\Q]h(X)\right)$.
\end{defi}
The classical height is invariant under the finite extension of number fields (cf. \cite[Lemma B.2.1]{Hindry}).

 We consider the Cayley form of $X$ defined in the Definition \ref{cayley form 2}. By the argument in Remark \ref{X is hypersurface I_X}, the system $I_{X,K}$ defined in Definition \ref{definition of wmu(I_X)} is generated by the polynomial $f(T_0,\ldots,T_n)$, and in fact $\overline{I}_X$ is a Hermitian line bundle over $\spec \O_K$ in this case. Then
\begin{equation}\label{explicit I_X when X is hypersurface}
\wmu(\overline{I}_X)=-\sum_{\p\in\spm\O_K}\frac{[K_v:\Q_v]}{[K:\Q]}\log\|f\|_{\p}-\sum_{v\in M_{K,\infty}}\frac{[K_v:\Q_v]}{[K:\Q]}\log\|f\|_{v,\mathrm{sym}},
\end{equation}
where $\|\ndot\|_{v,\mathrm{sym}}$ is the symmetric norm of the space $\sym^\delta(\sE_K^\vee)$.

For a projective hypersurface $X$, we will compare the two heights ($\wmu(\overline{I}_X)$ is considered as a height of $X$) mentioned above.
\begin{prop}\label{height na\"ive-slope}
   Let $\overline{\sE}=\left(\O_K^{\oplus(n+1)},(\|\ndot\|_v)_{v\in M_{K,\infty}}\right)$ be the Hermitian vector bundle over $\spec\O_K$ equipped with the norms determined in \eqref{l^2-norm}, $X$ be a projective hypersurface of degree $\delta$ of $\mathbb P(\sE_K)$, $h(X)$ be the classical height of $X$ defined in Definition \ref{classical height of hypersurface}, and $\wmu(\overline I_X)$ as in Definition \ref{definition of wmu(I_X)}. Then we have
\[h(X)-\frac{n}{2}\log(\delta+1)\leqslant-\wmu(\overline{I}_X)\leqslant h(X)+\frac{3n}{2}\log(\delta+1).\]
\end{prop}
\begin{proof}
 Let $v\in M_{K,\infty}$, and $s\in H^0\left(\mathbb P(\mathcal E_K),\O_{\mathbb P(\sE_K)}(\delta)\right)$ be the non-zero global section which defines the hypersurface $X$. By the equality \eqref{symmetric norm vs John norm} and Proposition \ref{symmetric norm vs John norm, constant}, we have
\begin{equation*}
  \log\|s\|_{v,J}=\log\|s\|_{v,\mathrm{sym}}+R_0(n,\delta),
\end{equation*}
where the constant $R_0(n,\delta)$ satisfies
\begin{equation*}
  0\leqslant R_0(n,\delta)\leqslant\sqrt{r(n,\delta)},
\end{equation*}
where $r(n,\delta)={n+\delta\choose n}$. In addition, we have the inequality
\begin{equation*}
  \|s\|_{v,\sup}\leqslant\|s\|_{v,J}\leqslant\sqrt{r(n,\delta)}\|s\|_{v,\sup}
\end{equation*}
by the equality \eqref{john norm}, where $\|\ndot\|_{v,J}$ is the norm of John induced from $\|\ndot\|_{v,\sup}$ defined in \eqref{definition of sup norm}. So we have the inequality
\begin{equation*}
  \log\|s\|_{v,\sup}-\frac{1}{2}\log r(n,\delta)\leqslant\log\|s\|_{v,\mathrm{sym}}\leqslant\log\|s\|_{v,\sup}+\frac{1}{2}\log r(n,\delta)
\end{equation*}
by Proposition \ref{slope of different norms}.

By definition, the norm $\|s\|_{v,\sup}=\sup\limits_{x\in\mathbb{P}(\mathcal E_{K,v})(\C)}\|s(x)\|_{v,\mathrm{FS}}$ corresponds to the Fubini-Study norm over $\mathbb{P}(\mathcal E_K)$ at the place $v\in M_{K,\infty}$, which is equal to $\frac{|s(x)|_v}{|x|_{v}^\delta}$, where $|\ndot|_{v}$ is the norm induced by the Hermitian norm over $\overline{\mathcal E}$. The value $\|s(x)\|_{v,\mathrm{FS}}$ does not depend on the choice of the projective coordinate of the point $x$.

In order to obtain an upper bound of $-\wmu(\overline{I}_X)$, we suppose that the hypersurface $X$ is defined by the non-zero homogeneous equation
    \[f(T_0,T_1,\ldots,T_n)=\sum_{\begin{subarray}{x}(i_0,\ldots,i_n)\in\mathbb N^{n+1}\\ i_0+\cdots+i_n=\delta\end{subarray}}a_{i_0,i_1,\ldots,i_n}T_0^{i_0}T_1^{i_1}\cdots T_n^{i_n}.\]
    Then for any place $v\in M_{K,\infty}$, we obtain
\begin{equation*}
  \sup\limits_{x\in \mathbb{P}(\mathcal E_K)(\C_v)}\frac{|v(f)(x)|_v}{|v(x)|_v^\delta}\leqslant{n+\delta\choose \delta}\max\limits_{\begin{subarray}{x}(i_0,\ldots,i_n)\in\mathbb N^{n+1}\\ i_0+\cdots+i_n=\delta\end{subarray}}\{|a_{i_0,\ldots,i_n}|_v\},
\end{equation*}
for there are at most ${n+\delta\choose \delta}$ non-zero terms in the equation $f(T_0,\ldots,T_n)=0$. Then we obtain
\begin{equation*}
  -\wmu(\overline{I}_X)\leqslant h(X)+\frac{3}{2}\log{n+\delta\choose \delta}\leqslant h(X)+\frac{3}{2}n\log(\delta+1),
\end{equation*}
where we use the trivial estimate ${n+\delta\choose \delta}\leqslant(\delta+1)^n$ at the last inequality above.

   Next, we will find a lower bound $-\wmu(\overline{I}_X)$. For every place $v\in M_{K,\infty}$, let $a_{\alpha_0,\ldots,\alpha_n}$ be one of the coefficients of $f(T_0,\ldots,T_n)$ such that $|a_{\alpha_0,\ldots,\alpha_n}|_v=\max\limits_{i_0+\cdots+i_n=\delta}\{|a_{i_0,\ldots, i_n}|_v\}$. By the integration formula of Cauchy, we have
\begin{equation*}
  \frac{1}{(2\pi i)^{n+1}}\int_{|z_0|_v=\cdots=|z_0|_v=1}f(z_0,\ldots,z_n)z_0^{-\alpha_0-1}\cdots z_n^{-\alpha_n-1}dz_0\cdots dz_n=a_{\alpha_0,\ldots,\alpha_n}.
\end{equation*}
So we obtain
\begin{eqnarray*}
  & &|a_{\alpha_0,\ldots,\alpha_n}|_v\\
  &=&\left|\frac{1}{(2\pi i)^{n+1}}\int_{|z_0|_v=\cdots=|z_0|_v=1}f(z_0,\ldots,z_n)z_0^{-\alpha_0-1}\cdots z_n^{-\alpha_n-1}dz_0\cdots dz_n\right|_v\\
  &=&\left|\int_{[0,1]^{n+1}}f(e^{2\pi it_0},\ldots,e^{2\pi it_n})e^{-2\pi it_0\alpha_0}\cdots e^{-2\pi it_n\alpha_n}dt_0\cdots dt_n\right|_v\\
  &\leqslant&\int_{[0,1]^{n+1}}\left|f(e^{2\pi it_0},\ldots,e^{2\pi it_n})\right|_vdt_0\cdots dt_n\\
  &\leqslant&\sup_{\begin{subarray}{c}x\in\C^{n+1}\\|x|\leqslant1\end{subarray}}\left|f(x)\right|_v\\
  &=&\sup\limits_{x\in\mathbb{P}(\mathcal E_K)(\C_v)}\frac{\left|f(x)\right|_v}{\left|x\right|^\delta_v}.
\end{eqnarray*}
Then we have
\begin{equation*}
  \log\|f\|_{v,\mathrm{sym}}\geqslant\max\limits_{i_0+\cdots+i_n=\delta}\{|a_{i_0,\ldots, i_n}|_v\}-\frac{1}{2}\log r(n,\delta)
\end{equation*}
for every place $v\in M_{K,\infty}$. Then we obtain
\begin{equation*}
  -\wmu(\overline{I}_X)\geqslant h(X)-\frac{1}{2}\log r(n,\delta)\geqslant h(X)-\frac{n}{2}\log(\delta+1),
\end{equation*}
which terminates the proof.
\end{proof}
\subsection{Counting rational points with small heights}
We keep all the notations and definitions in \S \ref{basic setting}. Let $X$ be a closed sub-scheme of $\mathbb P(\sE_K)$. and $Z=\left\{P_i\right\}_{i\in I}$ be a family of rational points of $X$. The evaluation map
\begin{equation*}
\eta_{Z,D}:\;E_{D,K}=H^0\left(\mathbb{P}(\mathcal{E}_K),\O(D)\right)\rightarrow \bigoplus_{i\in I}P_i^*\O(D)
\end{equation*}
can be factorized through $\eta_{X,D}$ defined in \eqref{evaluation map}. We denote by
\begin{equation}
  \phi_{Z,D}:\;F_{D,K}\rightarrow\bigoplus_{i\in I}P_i^*\O(D)
\end{equation}
the homomorphism such that $\phi_{Z,D}\circ\eta_{X,D}=\eta_{Z,D}$.

We have the following result.
\begin{prop}[\cite{Chen1}, Propoosition 2.12]\label{evaluation map of points}
  With all the notations above. If $X$ is integral, and we have the inequality
  \[\sup_{i\in I}h_{\overline{\O(1)}}(P_i)<\frac{\wmu_{\max}(\F_D)}{D}-\frac{1}{2D}\log r_1(n,D),\]
  where $r_1(n,D)=\rg(F_{D})$. Then the homomorphism $\phi_{Z,D}$ cannot be injective.
\end{prop}
The main tools to prove the above proposition is the slope inequalities, see \cite[Appendix A]{BostBour96}.

We combine \cite[Proposition 2.12]{Chen1} and the lower bound of $\wmu(\F_D)$ in Proposition \ref{numerical result of arithmetic hilbert of hypersurface} of hypersurfaces, and then we obtain the following result.
\begin{theo}\label{covered by one hypersurface}
  Let $K$ be a number field, and $X$ be an integral hypersurface of $\mathbb P^n_K$ of degree $\delta$. We suppose that $B$ is a positive real number satisfying the inequality
  \[\frac{\log B}{[K:\Q]}<\frac{1}{n\delta}h(X)+B_0(n,\delta)-\frac{1}{2}\log (n+1)-\frac{1}{2\delta}\log(\delta+1),\]
  where $B_0(n,\delta)$ is defined in Proposition \ref{numerical result of arithmetic hilbert of hypersurface}, and $h(X)$ is defined in Definition \ref{classical height of hypersurface}. Then the set $S(X;B)$ can be covered by a hypersurface of degree smaller than or equal to $\delta$ which does not contain the generic point of $X$, where $S(X;B)$ is defined in \eqref{S(X;B)}.
\end{theo}
\begin{proof}
If we have the inequality
\[\frac{\log B}{[K:\Q]}<\frac{1}{n\delta}h(X)+B_0(n,\delta)-\frac{1}{2}\log (n+1)-\frac{1}{2\delta}\log(\delta+1),\]
then by Proposition \ref{height na\"ive-slope}, we have
 \[\frac{\log B}{[K:\Q]}<-\frac{1}{n\delta}\wmu(\overline I_X)+B_0(n,\delta)-\frac{1}{2}\log (n+1),\]
 where $\wmu(\overline I_X)$ is defined in Definition \ref{definition of wmu(I_X)}. In addition, by Proposition \ref{numerical result of arithmetic hilbert of hypersurface}, we have
\[\frac{\log B}{[K:\Q]}<\frac{\wmu(\F_D)}{D}-\frac{1}{2}\log (n+1)\leqslant\frac{\wmu_{\max}(\F_D)}{D}-\frac{1}{2D}\log r_1(n,D)\]
 for every $D\geqslant\delta$. Then by Proposition \ref{evaluation map of points}, we have the assertion.
\end{proof}
\begin{rema}\label{explicit of covered by one hypersurface}
  With all the notations in Theorem \ref{covered by one hypersurface}. If a positive real number $B$ satisfies
  \[H_K(X)\gg_{n,K,\delta}B^{n\delta},\]
  then $S(X;B)$ can be covered by a hypersurface of degree smaller than or equal to $\delta$ which does not contain the generic point of $X$. The above implicit constant depending on $n$, $K$, and $\delta$ is original from Theorem \ref{covered by one hypersurface}, which can be explicitly written down. By \eqref{difference between classical height and arakelov height}, we use no matter the classical absolute logarithmic height defined in Definition \ref{weil height} or the Arakelov height introduced in Definition \ref{arakelov height}, we will always get the above result.
\end{rema}
\appendix
\section{An estimate of the constant $C(n,D)$}
The aim of this appendix it to give an explicit uniform estimate of the constant
 \begin{equation}\label{constant C-2}
C(n,D)=\sum\limits_{\begin{subarray}{c} i_0+\cdots+i_n=D \\ i_0,\ldots,i_n\geqslant0\end{subarray}}\log\left(\frac{i_0!\cdots i_n!}{D!}\right),
\end{equation}
defined in the equality \eqref{constant C}. In fact, we will prove (in Theorem \ref{estimate of C(n,D)})
\begin{eqnarray}\label{estimation of C-2}
C(n,D)&=&\frac{1-\mathcal H_{n+1}}{n!}D^{n+1}-\frac{n-2}{2n!}D^n\log D\\
& &+\frac{1}{n!}\Biggr(\left(-\frac{1}{6}n^3-\frac{3}{4}n^2-\frac{13}{12}n+2\right)\mathcal H_n\nonumber\\
& &\:+\frac{1}{4}n^3+\frac{17}{24}n^2+\left(\frac{119}{72}-\frac{1}{2}\log\left(2\pi\right)\right)n-4+\log\left(2\pi\right)\Biggr)D^n\nonumber\\
& &+o(D^n),\nonumber
\end{eqnarray}
where $\mathcal H_n=1+\frac{1}{2}+\cdots+\frac{1}{n}$. In addition, we will give both a uniform lower and upper bounds of the reminder explicitly. The only preliminary knowledge for this section is the single variable calculus.

In the rest of this section, we note $r(n,D)={n+D\choose n}$, and $C(n,D)$ same as in the equality \eqref{constant C-2}.
\subsection{Preliminaries}
In this part, we will give some preliminary calculation for the estimate of $C(n,D)$.
\begin{lemm}\label{3.3.3}
We have
\[r(n,D)=\sum_{m=0}^Dr(n-1,m).\]
\end{lemm}
\begin{proof}
In fact, we have
\[r(n,D)={n+D\choose n}=\sum_{m=0}^D{n+m-1\choose m}=\sum_{m=0}^Dr(n-1,m).\]
\end{proof}
\begin{lemm}\label{3.3.4}
We have
\[C(n,D)=\sum_{m=0}^D\left(C(n-1,m)+r(n-1,m)\log{D\choose m}^{-1}\right).\]
\end{lemm}
\begin{proof}
In fact, we have
\begin{eqnarray*}
  C(n,D)&=&\sum_{m=0}^D\left(\sum\limits_{\begin{subarray}{c} i_0+\cdots+i_{n-1}=D-m \\ i_0,\ldots,i_n\geqslant0\end{subarray}}\log\left(\frac{i_0!\cdots i_{n-1}!}{D!}\right)+r(n-1,m)\log m!\right)\\
&=&\sum_{m=0}^D\Biggr(\sum\limits_{\begin{subarray}{c} i_0+\cdots+i_{n-1}=D-m \\ i_0,\ldots,i_n\geqslant0\end{subarray}}\log\left(\frac{i_0!\cdots i_{n-1}!}{(D-m)!}\right)+r(n-1,m)\log\frac{(D-m)!}{D!}\\
& &+r(n-1,m)\log m!\Biggr)\\
&=&\sum_{m=0}^D\left(C(n-1,m)+r(n-1,m)\log{D\choose m}^{-1}\right).
\end{eqnarray*}
\end{proof}
Let
\[Q(n,D)=\sum_{m=0}^Dr(n-1,m)\log{D\choose m},\]
then we have
\begin{equation}\label{C(n,D)->Q(n,D)}
  C(n,D)=\sum_{m=0}^DC(n-1,m)-Q(n,D)
\end{equation}
by Lemma \ref{3.3.3} and Lemma \ref{3.3.4}. By definition, we obtain $C(0,D)\equiv0$. Then in order to estimate $C(n,D)$, we need to consider $Q(n,D)$.
\begin{lemm}
We have
\[Q(n,D)=\sum_{m=2}^D\Big(r(n,m-1)-r(n,D-m)\Big)\log m.\]
\end{lemm}
\begin{proof}
  By Abel transformation, we obtain,
\begin{eqnarray*}
  Q(n,D)&=&\sum_{m=1}^{D}r(n-1,m)\log{D\choose m}\\
&=&\sum_{m=1}^{D-1}\left(\sum_{k=1}^{m}r(n-1,k)\right)\left(\log{D\choose m}-\log{D\choose m+1}\right)\\
&=&\sum_{m=1}^{D-1}\Big(r(n,m)-1\Big)\log\frac{m+1}{D-m}.
\end{eqnarray*}
In addition, we have the equality
\[\sum_{m=1}^{D-1}r(n,m)\log(m+1)=\sum_{m=2}^Dr(n,m-1)\log m,\]
the equality
\[\sum_{m=1}^{D-1}r(n,m)\log(D-m)=\sum_{m=2}^{D-1}r(n,D-m)\log m,\]
and the equality
\[\sum_{m=1}^{D-1}\log\frac{m+1}{D-m}=\log D=r(n,0)\log D.\]
Then we obtain
\begin{eqnarray*}
  & &\sum_{m=1}^{D-1}\Big(r(n,m)-1\Big)\log\frac{m+1}{D-m}\\
&=&\sum_{m=2}^Dr(n,m-1)\log m-\sum_{m=2}^{D-1}r(n,D-m)\log m-r(n,D-D)\log D\\
&=&\sum_{m=2}^D\Big(r(n,m-1)-r(n,D-m)\Big)\log m,
\end{eqnarray*}
which terminates the proof.
\end{proof}
Let
\begin{equation}\label{def of S(n,D)}
  S(n,D)=\sum_{m=2}^D\Big((m-1)^n-(D-m)^n\Big)\log m.
\end{equation}
By the inequality
\[\frac{D^n}{n!}+\frac{(n+1)D^{n-1}}{2(n-1)!}\leqslant r(n,D)\leqslant\frac{D^n}{n!}+\frac{(n+1)D^{n-1}}{2(n-1)!}+(n-1)D^{n-2},\]
we obtain the following result.
\begin{prop}\label{upper bound and lower bound of Q(n,D)}
  Let $S(n,D)$ as in \eqref{def of S(n,D)}. Then we have
\[Q(n,D)\geqslant\frac{1}{n!}S(n,D)+\frac{n+1}{2(n-1)!}S(n-1,D)-(n-1)^2(D-1)^{n-1}\log D\]
and
\[Q(n,D)\leqslant\frac{1}{n!}S(n,D)+\frac{n+1}{2(n-1)!}S(n-1,D)+(n-1)^2(D-1)^{n-1}\log D.\]
\end{prop}
\begin{proof}
  In fact, we have
\begin{eqnarray*}
  Q(n,D)&\geqslant&\sum_{m=2}^D\Biggr(\frac{(m-1)^n}{n!}+\frac{(n+1)(m-1)^{n-1}}{2(n-1)!}-\frac{(D-m)^n}{n!}-\\
& &\frac{(n+1)(D-m)^{n-1}}{2(n-1)!}-(n-1)(D-m)^{n-2}\Biggr)\log m\\
&=&\frac{1}{n!}S(n,D)+\frac{n+1}{2(n-1)!}S(n-1,D)-(n-1)\sum_{m=2}^D(D-m)^{n-2}\log m\\
&\geqslant&\frac{1}{n!}S(n,D)+\frac{n+1}{2(n-1)!}S(n-1,D)-(n-1)^2(D-1)^{n-1}\log D
\end{eqnarray*}
and
\begin{eqnarray*}
  Q(n,D)&\leqslant&\sum_{m=2}^D\Biggr(\frac{(m-1)^n}{n!}+\frac{(n+1)(m-1)^{n-1}}{2(n-1)!}-\frac{(D-m)^n}{n!}-\\
& &\frac{(n+1)(D-m)^{n-1}}{2(n-1)!}+(n-1)(m-1)^{n-2}\Biggr)\log m\\
&=&\frac{1}{n!}S(n,D)+\frac{n+1}{2(n-1)!}S(n-1,D)+(n-1)\sum_{m=2}^D(m-1)^{n-2}\log m\\
&\leqslant&\frac{1}{n!}S(n,D)+\frac{n+1}{2(n-1)!}S(n-1,D)+(n-1)^2(D-1)^{n-1}\log D,
\end{eqnarray*}
which terminates the proof.
\end{proof}
\subsection{Explicit estimate of $S(n,D)$ when $n\geqslant2$}
We fix a real number $\epsilon\in]0,\frac{1}{6}[$. Let
\begin{equation}\label{def of S_1(n,D)}
  S_1(n,D)=\sum_{2\leqslant m\leqslant D^{1/2+\epsilon}}\Big((m-1)^n-(D-m)^n\Big)\log m
\end{equation}
and
\begin{equation}\label{def of S_2(n,D)}
  S_2(n,D)=\sum_{D^{1/2+\epsilon}<m\leqslant D}\Big((m-1)^n-(D-m)^n\Big)\log m,
\end{equation}
then we have
\[S(n,D)=S_1(n,D)+S_2(n,D),\]
where $S(n,D)$ is defined in \eqref{def of S(n,D)}.

For estimating $S(n,D)$, we need an upper bound and a lower bound of $S_1(n,D)$ and $S_2(n,D)$ respectively.

First, we are going to estimate $S_1(n,D)$. In fact, we have
\[0\leqslant\sum_{2\leqslant m\leqslant D^{1/2+\epsilon}}(m-1)^n\log m\leqslant\frac{1}{2}D^{(1/2+\epsilon)(n+1)}\log D.\]
By the choice of $\epsilon$ and $n$, we have $(1/2+\epsilon)(n+1)<n$. In addition, we have $2\leqslant m\leqslant D^{1/2+\epsilon}$, so we have
\[D^n-nD^{n-1}m\leqslant(D-m)^n\leqslant D^n-nD^{n-1}m+\frac{(n-1)2^{n-1}e}{\pi\sqrt{n}}D^{n-2}m^2.\]
Then we obtain
\[\sum_{2\leqslant m\leqslant D^{1/2+\epsilon}}(D-m)^n\log m\geqslant\sum_{2\leqslant m\leqslant D^{1/2+\epsilon}}(D^n-nD^{n-1}m)\log m\]
and
\begin{eqnarray*}
\sum_{2\leqslant m\leqslant D^{1/2+\epsilon}}(D-m)^n\log m&\leqslant&\sum_{2\leqslant m\leqslant D^{1/2+\epsilon}}\Big(D^n-nD^{n-1}m\Big)\log m+D^{1/2+\epsilon}\\
& & +\frac{(n-1)2^{n-1}e}{\pi\sqrt{n}}D^{n-1/2+3\epsilon}\log D,
\end{eqnarray*}
where we have $n-1/2+3\epsilon<n$ by the choice of $\epsilon$.

By the above argument, we obtain:
\begin{prop}\label{3.3.7}
Let $S_1(n,D)$ be as in \eqref{def of S_1(n,D)}. We have
\[S_1(n,D)=D^{n}\left(\sum_{2\leqslant m\leqslant D^{1/2+\epsilon}}\log m\right)-nD^{n-1}\left(\sum_{2\leqslant m\leqslant D^{1/2+\epsilon}}m\log m\right)+o(D^n).\]
In addition, we have
\[S_1(n,D)\geqslant D^{n}\left(\sum_{2\leqslant m\leqslant D^{1/2+\epsilon}}\log m\right)-nD^{n-1}\left(\sum_{2\leqslant m\leqslant D^{1/2+\epsilon}}m\log m\right)\]
and
\begin{eqnarray*}
  S_1(n,D)&\leqslant& D^{n}\left(\sum_{2\leqslant m\leqslant D^{1/2+\epsilon}}\log m\right)-nD^{n-1}\left(\sum_{2\leqslant m\leqslant D^{1/2+\epsilon}}m\log m\right)\\
  & &+\frac{1}{2}D^{(1/2+\epsilon)(n+1)}\log D+D^{1/2+\epsilon}+\frac{(n-1)2^{n-1}e}{\pi\sqrt{n}}D^{n-1/2+3\epsilon}\log D.
\end{eqnarray*}
\end{prop}

In order to estimate $S_2(n,D)$, we are going to introduce the following lemma. It is a simple form of Euler-Maclaurin formula.
\begin{lemm}\label{integral approximation}
  Let $p,q$ be two positive integers, where $p\leqslant q$. For any function $f\in C^2([p-\frac{1}{2},q+\frac{1}{2}])$, there exists a real number $\Theta$ such that
\[\sum_{m=p}^qf(m)=\int_{p-\frac{1}{2}}^{q+\frac{1}{2}}f(x)dx+\frac{1}{8}f'\left(p-\frac{1}{2}\right)-\frac{1}{8}f'\left(q+\frac{1}{2}\right)+\Theta,\]
where $|\Theta|\leqslant(q-p+1)\sup\limits_{p-1/2\leqslant x\leqslant q+1/2}|f''(x)|$.
\end{lemm}
\begin{proof}
 By definition, we have
\begin{eqnarray*}
& &\int_{p-\frac{1}{2}}^{q+\frac{1}{2}}f(x)dx+\frac{1}{8}f'\left(p-\frac{1}{2}\right)-\frac{1}{8}f'\left(q+\frac{1}{2}\right)\\
&=&\sum_{m=p}^q\left(\int_{m-\frac{1}{2}}^{m+\frac{1}{2}}f(x)dx+\frac{1}{8}f'\left(m-\frac{1}{2}\right)-\frac{1}{8}f'\left(m+\frac{1}{2}\right)\right).
\end{eqnarray*}
Then we need to prove
\[f(m)=\int_{m-\frac{1}{2}}^{m+\frac{1}{2}}f(x)dx+\frac{1}{8}f'\left(m-\frac{1}{2}\right)-\frac{1}{8}f'\left(m+\frac{1}{2}\right)+\Theta(m),\]
where
\[\Theta(m)\leqslant\sup\limits_{m-1/2\leqslant x\leqslant m+1/2}|f''(x)|.\]

For a real number $x\in[m-\frac{1}{2},m]$, let
\[g(x)=f(x)-f(m)-f'\left(m-\frac{1}{2}\right)(x-m).\]
Then we have
\[g(m)=0,\:g'\left(m-\frac{1}{2}\right)=0,\:g''(x)=f''(x).\]

For the function $g(x)$, we have
\[\left|\int_{m-\frac{1}{2}}^mg(x)dx\right|\leqslant\frac{1}{2}\sup|g(x)|\leqslant\frac{1}{2}\sup|g''(x)|=\frac{1}{2}\sup|f''(x)|.\]
Then we obtain
\[\left|\int_{m-\frac{1}{2}}^mf(x)dx-\frac{1}{2}f(m)+\frac{1}{8}f'\left(m-\frac{1}{2}\right)\right|\leqslant\frac{1}{2}\sup|f''(x)|.\]
By the similar argument, we obtain
\[\left|\int_m^{m+\frac{1}{2}}f(x)dx-\frac{1}{2}f(m)-\frac{1}{8}f'\left(m+\frac{1}{2}\right)\right|\leqslant\frac{1}{2}\sup|f''(x)|.\]
Then we have
\[\left|\int_{m-\frac{1}{2}}^{m+\frac{1}{2}}f(x)dx-f(m)+\frac{1}{8}f'\left(m-\frac{1}{2}\right)-\frac{1}{8}f'\left(m+\frac{1}{2}\right)\right|\leqslant\sup|f''(x)|,\]
which proves the assertion.
\end{proof}
Let $x$ be a real number. We denote by $[x]_+$ the smallest integer which is larger than $x$. Let
\[f(x)=\Big((x-1)^n-(D-x)^n\Big)\log x,\]
where $[D^{1/2+\epsilon}]_+-\frac{1}{2}\leqslant x\leqslant D+\frac{1}{2}$.
\begin{prop}
Let $S_2(n,D)$ be as in \eqref{def of S_2(n,D)}. We have
\[S_2(n,D)=\int_{[D^{1/2+\epsilon}]_+-\frac{1}{2}}^{D+\frac{1}{2}}\Big((x-1)^n-(D-x)^n\Big)\log x dx+o(D^n).\]
In addition, we have
\begin{eqnarray*}
  S_2(n,D)&\geqslant& \int_{[D^{1/2+\epsilon}]_+-\frac{1}{2}}^{D+\frac{1}{2}}\Big((x-1)^n-(D-x)^n\Big)\log x dx\\
  & &-8n(n-1)\left(D-\frac{1}{2}\right)^{n-1}\log\left(D+\frac{1}{2}\right),
\end{eqnarray*}
and
\begin{eqnarray*}
  S_2(n,D)&\leqslant& \int_{[D^{1/2+\epsilon}]_+-\frac{1}{2}}^{D+\frac{1}{2}}\Big((x-1)^n-(D-x)^n\Big)\log x dx\\
  & &+8n(n-1)\left(D-\frac{1}{2}\right)^{n-1}\log\left(D+\frac{1}{2}\right).
\end{eqnarray*}
\end{prop}
\begin{proof}
 The estimate of the dominant term of $S_2(n,D)$ is by Lemma \ref{integral approximation}. For the estimate of the remainder of $S_2(n,D)$, we have
\[f'(x)=\frac{(x-1)^n-(D-x)^n}{x}+n\left((x-1)^{n-1}+(D-x)^{n-1}\right)\log x\]
and
\begin{eqnarray*}
  f''(x)&=&-\frac{(x-1)^n-(D-x)^n}{x^2}+\frac{2n\left((x-1)^{n-1}+(D-x)^{n-1}\right)}{x}\\
  & &+n(n-1)\left((x-1)^{n-2}+(D-x)^{n-2}\right)\log x.
\end{eqnarray*}
Then we obtain that
\[\left|\frac{1}{8}f'\left([D^{1/2+\epsilon}]_+-\frac{1}{2}\right)-\frac{1}{8}f'\left(D+\frac{1}{2}\right)+\Big(D-D^{1/2+\epsilon}+1\Big)\sup_{[D^{1/2+\epsilon}]_+-\frac{1}{2}\leqslant x\leqslant D+\frac{1}{2}}|f''(x)|\right|\]
is smaller than or equal to
\[8n(n-1)\left(D-\frac{1}{2}\right)^{n-1}\log\left(D+\frac{1}{2}\right).\]
So we have the result.
\end{proof}

We will estimate $S_2(n,D)$ by some integrations. In fact, we have
\[\int_1^{[D^{1/2+\epsilon}]_+-\frac{1}{2}}(x-1)^n\log x dx\geqslant0,\]
and
\[\int_1^{[D^{1/2+\epsilon}]_+-\frac{1}{2}}(x-1)^n\log x dx\leqslant\frac{1}{2}\left(D^{(1/2+\epsilon)(n+1)}\log D\right).\]
We consider the integration
\[\int_1^{[D^{1/2+\epsilon}]_+-\frac{1}{2}}(D-x)^n\log x dx-\int_1^{[D^{1/2+\epsilon}]_+-\frac{1}{2}}\left(D^n-nD^{n-1}x\right)\log x dx.\]
Then we have
\begin{eqnarray*}
& &\int_1^{[D^{1/2+\epsilon}]_+-\frac{1}{2}}(D-x)^n\log x dx-\int_1^{[D^{1/2+\epsilon}]_+-\frac{1}{2}}\left(D^n-nD^{n-1}x\right)\log x dx\\
&\geqslant&0,
\end{eqnarray*}
and
\begin{eqnarray*}
& &\int_1^{[D^{1/2+\epsilon}]_+-\frac{1}{2}}\log x(D-x)^ndx-\int_1^{[D^{1/2+\epsilon}]_+-\frac{1}{2}}\log x\left(D^n-nD^{n-1}x\right)dx\\
&\leqslant&\frac{(n-1)2^{n-1}e}{\pi\sqrt{n}}D^{n-1/2+3\epsilon}\log D.
\end{eqnarray*}
So we obtain:
\begin{coro}\label{3.3.10}
We have
\begin{eqnarray*}
  S_2(n,D)&=&\int_{1}^{D+\frac{1}{2}}\Big((x-1)^n-(D-x)^n\Big)\log x dx\\
& &-D^n\left(\int_{1}^{[D^{1/2+\epsilon}]_+-\frac{1}{2}}\log xdx\right)+nD^{n-1}\left(\int_{1}^{[D^{1/2+\epsilon}]_+-\frac{1}{2}}x\log xdx\right)+o(D^n).
\end{eqnarray*}
In addition, we have
\begin{eqnarray*}
  S_2(n,D)&\geqslant&\int_{1}^{D+\frac{1}{2}}\Big((x-1)^n-(D-x)^n\Big)\log x dx\\
& &-D^n\left(\int_{1}^{[D^{1/2+\epsilon}]_+-\frac{1}{2}}\log xdx\right)+nD^{n-1}\left(\int_{1}^{[D^{1/2+\epsilon}]_+-\frac{1}{2}}x\log xdx\right)\\
& &-8n(n-1)\left(D-\frac{1}{2}\right)^{n-1}\log\left(D+\frac{1}{2}\right)
\end{eqnarray*}
and
\begin{eqnarray*}
  S_2(n,D)&\leqslant&\int_{1}^{D+\frac{1}{2}}\Big((x-1)^n-(D-x)^n\Big)\log x dx\\
& &-D^n\left(\int_{1}^{[D^{1/2+\epsilon}]_+-\frac{1}{2}}\log xdx\right)+nD^{n-1}\left(\int_{1}^{[D^{1/2+\epsilon}]_+-\frac{1}{2}}x\log xdx\right)\\
& &+8n(n-1)\left(D-\frac{1}{2}\right)^{n-1}\log\left(D+\frac{1}{2}\right)+\frac{(n-1)2^{n-1}e}{\pi\sqrt{n}}D^{n-1/2+3\epsilon}\log D.
\end{eqnarray*}
\end{coro}
We are going to combine the estimates of $S_1(n,D)$ and $S_2(n,D)$ in Proposition \ref{3.3.7} and Corollary \ref{3.3.10}. First, we have:
\begin{lemm}\label{gamma}
The function
\[\sum_{m\leqslant D^{1/2+\epsilon}}\log m-\int_1^{[D^{1/2+\epsilon}]_+-\frac{1}{2}}\log xdx\]
converges to
  \[-1+\frac{1}{2}\log\left(2\pi\right)\]
when $D$ tends to $+\infty$. In addition, we have
\begin{eqnarray*}
0&\leqslant&\sum_{m\leqslant D^{1/2+\epsilon}}\log m-\int_1^{[D^{1/2+\epsilon}]_+-\frac{1}{2}}\log xdx-\left(-1+\frac{1}{2}\log\left(2\pi\right)\right)\\
&\leqslant&\frac{3}{2}\log\frac{3}{2}+\frac{1}{2}-\frac{1}{2}\log\left(2\pi\right)
\end{eqnarray*}
\end{lemm}
\begin{proof}
  Let
\[a_n=\sum_{m\leqslant n}\log m-\int_1^{n+\frac{1}{2}}\log xdx.\]
By definition, the series $\{a_n\}_{n\geqslant1}$ is decreasing, and $a_1=\frac{3}{2}\log\frac{3}{2}-\frac{1}{2}$.

By the Stirling formula
\[n!=\sqrt{2\pi n}\left(\frac{n}{e}\right)^n\left(1+O\left(\frac{1}{n}\right)\right),\]
we obtain
\[\sum_{m\leqslant n}\log m=\log(n!)=\frac{1}{2}\log\left(2\pi\right)+\frac{1}{2}\log n+n\log n-n+o(1).\]
Next, we consider the integration
\begin{eqnarray*}
\int_1^{n+\frac{1}{2}}\log xdx&=&\left(n+\frac{1}{2}\right)\log\left(n+\frac{1}{2}\right)-\left(n-\frac{1}{2}\right)\\
&=&\frac{1}{2}\log n+n\log n-n+1+o(1).
\end{eqnarray*}
Then we obtain the limit of $\{a_n\}_{n\geqslant1}$. So we have the assertion.
\end{proof}

\begin{lemm}\label{3.3.12}
We have
\[0\leqslant\sum_{m\leqslant D^{1/2+\epsilon}}m\log m-\int_1^{[D^{1/2+\epsilon}]_+-\frac{1}{2}}x\log xdx\leqslant\frac{1}{4}\log D.\]
\end{lemm}
\begin{proof}
In fact, we have
\[\sum_{m\leqslant D^{1/2+\epsilon}}m\log m-\int_1^{[D^{1/2+\epsilon}]_+-\frac{1}{2}}x\log xdx\geqslant0\]
by a direct calculation.

For the other side, we apply Lemma \ref{integral approximation} to the function $f(x)=\log x$, then we obtain
\[\left|\log m-\int_{m-\frac{1}{2}}^{m+\frac{1}{2}}\log xdx-\frac{1}{8}\left(m-\frac{1}{2}\right)^{-1}+\frac{1}{8}\left(m+\frac{1}{2}\right)^{-1}\right|\leqslant \left(m-\frac{1}{2}\right)^{-2}.\]
Then
\[\sum_{m\leqslant D^{1/2+\epsilon}}m\log m-\int_1^{[D^{1/2+\epsilon}]_+-\frac{1}{2}}x\log xdx\leqslant\frac{1}{4}\sum_{m\leqslant D^{1/2+\epsilon}}m^{-1}\leqslant\frac{1}{4}\log D.\]
\end{proof}
We combine the Corollary \ref{3.3.10}, Lemma \ref{gamma}, and Lemma \ref{3.3.12}, we obtain the following result.
\begin{prop}\label{integral of S(n,D)}
Let $S_n(n,D)$ be as in the equality \eqref{def of S(n,D)}. Then we have
\[S(n,D)=\int_1^{D+\frac{1}{2}}\Big((x-1)^n-(D-x)^n\Big)\log x dx+\left(-1+\frac{1}{2}\log\left(2\pi\right)\right) D^n+o(D^n).\]
In addition, we have
\begin{eqnarray*}
  & &S(n,D)-\left(\int_1^{D+\frac{1}{2}}\Big((x-1)^n-(D-x)^n\Big)\log x dx+\left(-1+\frac{1}{2}\log\left(2\pi\right)\right) D^n\right)\\
  &\geqslant&-8n(n-1)\left(D-\frac{1}{2}\right)^{n-1}\log\left(D+\frac{1}{2}\right)-\left(\frac{3}{2}\log\frac{3}{2}+\frac{1}{2}-\frac{1}{2}\log\left(2\pi\right)\right)
\end{eqnarray*}
and
\begin{eqnarray*}
  & &S(n,D)-\left(\int_1^{D+\frac{1}{2}}\Big((x-1)^n-(D-x)^n\Big)\log x dx+\left(-1+\frac{1}{2}\log\left(2\pi\right) \right)D^n\right)\\
  &\leqslant&8n(n-1)\left(D-\frac{1}{2}\right)^{n-1}\log\left(D+\frac{1}{2}\right)+\frac{(n-1)2^{n}e}{\pi\sqrt{n}}D^{n-1/2+3\epsilon}\log D.
\end{eqnarray*}
\end{prop}
In order to obtain an explicit estimate of $S(n,D)$, we have the following result:
\begin{prop}\label{bound of S(n,D)}
Let
\[\mathcal H_n=1+\frac{1}{2}+\cdots+\frac{1}{n},\]
and
\[A_1(n,D)=-2^{n+3}\left(D+\frac{1}{2}\right)^{n-1}\log\left(D+\frac{1}{2}\right)-\left(\frac{3}{2}\log\frac{3}{2}+\frac{1}{2}-\frac{1}{2}\log\left(2\pi\right)\right),\]
and
\[A'_1(n,D)=9n(n-1)\left(D+\frac{1}{2}\right)^{n-1}\log\left(D+\frac{1}{2}\right)+\frac{(n-1)2^{n}e}{\pi\sqrt{n}}D^{n-1/2+3\epsilon}\log D.\]
Then we have
\[S(n,D)\geqslant\frac{\mathcal H_{n}D^{n+1}}{n+1}-\frac{D^n\log D}{2}+\left(-1+\frac{1}{2}\log\left(2\pi\right)-\frac{1}{2n}\right)D^n+A_1(n,D)\]
and
\[S(n,D)\leqslant\frac{\mathcal H_{n}D^{n+1}}{n+1}-\frac{D^n\log D}{2}+\left(-1+\frac{1}{2}\log\left(2\pi\right)-\frac{1}{2n}\right)D^n+A'_1(n,D).\]

\end{prop}
\begin{proof}
   For estimating the dominant terms, we are going to calculate the coefficients of $D^{n+1}\log D$, $D^{n+1}$, $D^n\log D$ and $D^n$ in the integration in Proposition \ref{integral of S(n,D)}. For the integration in Proposition \ref{integral of S(n,D)}, we have the inequality
\begin{eqnarray*}
  & &\int_1^{D+\frac{1}{2}}\Big((x-1)^n-(D-x)^n\Big)\log x dx+\left(-1+\frac{1}{2}\log\left(2\pi\right) \right)D^n\\
&=&\frac{\left(D-\frac{1}{2}\right)^{n+1}\log\left(D+\frac{1}{2}\right)}{n+1}+\frac{\log\left(D+\frac{1}{2}\right)}{(-2)^{n+1}(n+1)}\\
& &-\int_{1}^{D+\frac{1}{2}}\frac{(x-1)^{n+1}+(D-x)^{n+1}}{(n+1)x}dx+\left(-1+\frac{1}{2}\log\left(2\pi\right) \right) D^n.
\end{eqnarray*}
For the integration $\int_{1}^{D+\frac{1}{2}}\frac{(x-1)^{n+1}+(D-x)^{n+1}}{(n+1)x}dx$, we have
\begin{eqnarray*}
  & &\int_{1}^{D+\frac{1}{2}}\frac{(x-1)^{n+1}+(D-x)^{n+1}}{(n+1)x}dx\\
&=&\frac{1}{n+1}\int_{1}^{D+\frac{1}{2}}\Biggr(\sum_{k=1}^{n+1}{n+1\choose k}x^{k-1}(-1)^{n-k+1}-\sum_{k=1}^{n+1}{n+1\choose k}(-x)^{k-1}D^{n-k+1}\\
& &+\frac{(-1)^{n+1}}{x}+\frac{D^{n+1}}{x}\Biggr)dx\\
&=&\frac{1}{n+1}\Biggr(\sum_{k=1}^{n+1}{n+1\choose k}\frac{\left(D+\frac{1}{2}\right)^{k}-1}{k}(-1)^{n-k+1}+\sum_{k=1}^{n+1}{n+1\choose k}\frac{\left(-D-\frac{1}{2}\right)^{k}+(-1)^k}{k}D^{n-k+1}\\
& &(-1)^{n+1}\log \left(D+\frac{1}{2}\right)+D^{n+1}\log \left(D+\frac{1}{2}\right)\Biggr),
\end{eqnarray*}
then we obtain that the coefficients of $D^{n+1}\log D$ is $0$.

For the coefficient of $D^{n+1}$, it is equal to
\begin{eqnarray*}
  & &-\frac{1}{(n+1)^2}-\frac{1}{n+1}\sum_{k=1}^{n+1}\frac{(-1)^k}{k}{n+1\choose k}\\
&=&-\frac{1}{(n+1)^2}+\frac{1}{n+1}\int_0^1\frac{(1-x)^{n+1}-1}{x}dx\\
&=&-\frac{1}{(n+1)^2}-\frac{1}{n+1}\int_0^1\left((1-x)^n+\cdots+1\right)dx\\
&=&-\frac{1}{(n+1)^2}+\frac{1}{n+1}\left(1+\frac{1}{2}+\cdots+\frac{1}{n+1}\right)\\
&=&\frac{1}{n+1}\left(1+\frac{1}{2}+\cdots+\frac{1}{n}\right).
\end{eqnarray*}
The coefficient of $D^n\log (D+\frac{1}{2})$ is equal to
\[-\frac{1}{2}.\]
The coefficient of $D^n$ is equal to
\begin{eqnarray*}& &-1+\frac{1}{2}\log\left(2\pi\right)-\frac{1}{2(n+1)}\sum_{k=1}^{n+1}{n+1\choose k}(-1)^k-\frac{n+1}{2(n+1)^2}-\frac{1}{2(n+1)}{n+1\choose n}\frac{1}{n}\\
&=&-1+\frac{1}{2}\log\left(2\pi\right)-\frac{1}{2n}.
\end{eqnarray*}

Next, we are going to estimate the remainder. We consider the estimate
\begin{eqnarray*}
  I(n,D)&:= &\int_1^{D+\frac{1}{2}}\Big((x-1)^n-(D-x)^n\Big)\log x dx+\left(-1+\frac{1}{2}\log\left(2\pi\right)\right) D^{n}\\
  & &-\frac{\mathcal H_{n}}{n+1}D^{n+1}+\frac{1}{2}D^n\log D-\left(-1+\frac{1}{2}\log\left(2\pi\right)-\frac{1}{2n}\right)D^n.
\end{eqnarray*}
We can confirm that
\[I(n,D)\leqslant\frac{n}{4}\left(D+\frac{1}{2}\right)^{n-1}\log\left(D+\frac{1}{2}\right)\]
and
\[I(n,D)\geqslant-2^n\left(D+\frac{1}{2}\right)^{n-1}\log\left(D+\frac{1}{2}\right).\]
We combine the above estimate of the integration $I(n,D)$ with the estimate of remainder in Proposition \ref{integral of S(n,D)}, we obtain that $A_1(n,D)$ and $A'_1(n,D)$ satisfy the requirement.
\end{proof}

\subsection{Estimate of $C(1,D)$}
Let
\begin{eqnarray*}
A_2(n,D)&=&\frac{A_1(n,D)}{n!}-\frac{(n+1)D^{n-1}\log D}{4(n-1)!}+\frac{(n+1)\left(-1+\frac{1}{2}\log\left(2\pi\right)-\frac{1}{2n}\right)}{2(n-1)!}D^{n-1}\\
& &+\frac{(n+1)A_1(n-1,D)}{2(n-1)!}-(n-1)^2(D-1)^{n-1}\log D,
\end{eqnarray*}
and
\begin{eqnarray*}
A'_2(n,D)&=&\frac{A'_1(n,D)}{n!}-\frac{(n+1)D^{n-1}\log D}{4(n-1)!}+\frac{(n+1)\left(-1+\frac{1}{2}\log\left(2\pi\right)-\frac{1}{2n}\right)}{2(n-1)!}D^{n-1}\\
& &+\frac{(n+1)A'_1(n-1,D)}{2(n-1)!}+(n-1)^2(D-1)^{n-1}\log D,
\end{eqnarray*}
where the constants $A_1(n,D)$ and $A'_1(n,D)$ are defined in Proposition \ref{bound of S(n,D)}. Then we have $A_2(n,D)\sim o(D^n)$ and $A'_2(n,D)\sim o(D^n)$. By Proposition \ref{upper bound and lower bound of Q(n,D)} and Proposition \ref{bound of S(n,D)}, for $n\geqslant2$, we obtain
\begin{eqnarray}\label{Q(n,D)lower bound}
Q(n,D)&\geqslant&\frac{\mathcal H_n D^{n+1}}{(n+1)!}-\frac{1}{2n!}D^n\log D\\
& &+\frac{1}{n!}\left(-1+\frac{1}{2}\log\left(2\pi\right)-\frac{1}{2n}+\frac{(n+1)\mathcal H_{n-1}}{2}\right)D^n+A_2(n,D),\nonumber
\end{eqnarray}
and
\begin{eqnarray}\label{Q(n,D)upper bound}
Q(n,D)&\leqslant&\frac{\mathcal H_n D^{n+1}}{(n+1)!}-\frac{1}{2n!}D^n\log D\\
& &+\frac{1}{n!}\left(-1+\frac{1}{2}\log\left(2\pi\right)-\frac{1}{2n}+\frac{(n+1)\mathcal H_{n-1}}{2}\right)D^n+A'_2(n,D).\nonumber
\end{eqnarray}

By the definition of $C(n,D)$ in \eqref{constant C-2}, we have $C(0,D)\equiv0$ for all $D\geqslant0$, and $C(n,0)\equiv C(n,1)\equiv0$ for any $n\geqslant0$. By the relation
\[C(n,D)=\sum_{m=0}^DC(n-1,m)-Q(n,D)\]
showed in \eqref{C(n,D)->Q(n,D)}, we need to calculate $C(1,D)$ for $D\geqslant2$ in order to estimate $C(n,D)$. By definition, we have
\[C(1,D)=-\log\prod_{m=0}^D{D\choose m}=-Q(1,D).\]
We are going to calculate $Q(1,D)$ for all $D\geqslant2$ directly. First, we have
\[Q(1,D)=\sum_{m=2}^D(m-D+m-1)\log m=2\sum_{m=2}^Dm\log m-\left(D+1\right)\sum_{m=2}^D\log m.\]
\begin{prop}\label{C(1,D)}
Let
\begin{eqnarray*}
  A_3(D)&=&a_3(D)+\frac{1}{8}\left(\log\frac{3}{2}+1\right)-\frac{1}{8}\left(\log\left([D^{1/2}]+\frac{1}{2}\right)+1\right)\\
  & &+\frac{1}{8}\left(\log\left([\sqrt{D}]+\frac{1}{2}\right)+1\right)-\frac{1}{8}\left(\log\left(D+\frac{1}{2}\right)+1\right),
\end{eqnarray*}
and
\begin{eqnarray*}
  A'_3(D)&=&a_3(D)+\frac{1}{8}\left(\log\frac{3}{2}+1\right)-\frac{1}{8}\left(\log\left([D^{1/2}]+\frac{1}{2}\right)+1\right)+\frac{2\sqrt{D}}{3}\\
  & &+\frac{1}{8}\left(\log\left([\sqrt{D}]+\frac{1}{2}\right)+1\right)-\frac{1}{8}\left(\log\left(D+\frac{1}{2}\right)+1\right)+\sqrt{D}\\
  & &+\frac{1}{4}+\frac{\pi^2}{6},
\end{eqnarray*}
where $a_3(D)\sim o(D)$ is given explicitly in the proof below. Then we have
\[Q(1,D)=2\int_{\frac{3}{2}}^{D+\frac{1}{2}}x\log xdx-(D+1)\int_{\frac{3}{2}}^{D+\frac{1}{2}}\log xdx-(D+1)\left(-1+\frac{1}{2}\log\left(2\pi\right)\right)+o(D).\]
In addition, we have
\[C(1,D)\geqslant-\frac{1}{2}D^2+\frac{1}{2}D\log D+\left(-1+\frac{1}{2}\log\left(2\pi\right) \right)D+A_3(D)\]
and
\[C(1,D)\leqslant-\frac{1}{2}D^2+\frac{1}{2}D\log D+\left(-1+\frac{1}{2}\log\left(2\pi\right) \right) D+A'_3(D).\]
\end{prop}
\begin{proof}
   For the sum $\sum\limits_{m=2}^Dm \log m$, we device it into two parts: the sum $\sum\limits_{m=2}^{[\sqrt{D}]}m \log m$ and the sum $\sum\limits_{m=[\sqrt{D}]+1}^Dm \log m$, where $[x]$ is the largest integer which is smaller than $x$.

For estimating $\sum\limits_{m=2}^{[\sqrt{D}]}m \log m$, by Lemma \ref{integral approximation}, we have
\[\sum_{m=2}^{[\sqrt{D}]}m \log m=\int_{\frac{3}{2}}^{[D^{1/2}]+1/2}x\log xdx+\frac{1}{8}\left(\log\frac{3}{2}+1\right)-\frac{1}{8}\left(\log\left([D^{1/2}]+\frac{1}{2}\right)+1\right)+\Theta_1,\]
where
\[0\leqslant|\Theta_1|\leqslant \sup\limits_{2/3\leqslant m\leqslant \sqrt{D}-1/2}\frac{\sqrt{D}-1}{m}\leqslant\frac{2\sqrt{D}}{3}.\]
 In addition, we have
 \[\int_{\frac{3}{2}}^{[D^{1/2}]+1/2}x\log xdx\sim \frac{1}{4}D\log D+\frac{1}{4}D+o(D).\]

For the sum $\sum\limits_{m=[\sqrt{D}]+1}^Dm \log m$, also by Lemma \ref{integral approximation}, we have
\begin{eqnarray*}
  & &\sum_{m=[\sqrt{D}]+1}^Dm \log m\\
  &=&\int_{[D^{1/2}]+1/2}^{D+1/2}x\log xdx+\frac{1}{8}\left(\log\left([\sqrt{D}]+\frac{1}{2}\right)+1\right)-\frac{1}{8}\left(\log\left(D+\frac{1}{2}\right)+1\right)+\Theta_2,
\end{eqnarray*}
where
\[0\leqslant|\Theta_2|\leqslant \sup\limits_{\sqrt{D}-1/2\leqslant m\leqslant D+1/2}\frac{D-\sqrt{D}+1}{m}\leqslant\sqrt{D}.\]

The estimate $\sum\limits_{m=2}^D\log m$ is by Lemma \ref{gamma}.

For an explicit calculation, we have
\begin{eqnarray*}
  & &2\int_{\frac{3}{2}}^{D+\frac{1}{2}}x\log xdx-(D+1)\int_{1}^{D+\frac{1}{2}}\log xdx-(D+1)\left(-1+\frac{1}{2}\log\left(2\pi\right)\right)\\
&=&\left(D+\frac{1}{2}\right)^2\log\left(D+\frac{1}{2}\right)-\frac{9}{4}\log\frac{3}{2}-\frac{1}{2}\left(D+\frac{1}{2}\right)^2+\frac{1}{2}\left(\frac{3}{2}\right)^2\\
& &-(D+1)\left(\left(D+\frac{1}{2}\right)\log\left(D+\frac{1}{2}\right)-\left(D+\frac{1}{2}\right)+1\right)-(D+1)\left(-1+\frac{1}{2}\log\left(2\pi\right)\right)\\
&=&\frac{1}{2}D^2-\frac{1}{2}D\log D-\left(-1+\frac{1}{2}\log\left(2\pi\right) \right) D+a_3(D),
\end{eqnarray*}
where $a_3(D)$ is the remainder of the above sum. By Lemma \ref{gamma}, we obtain that the constants $A_3(D)$ and $A'_3(D)$ in the assertion satisfy the requirement, for $C(1,D)=-Q(1,D)$.
\end{proof}
\subsection{Estimate of $C(n,D)$}
In this part, we will estimate the constant $C(n,D)$ as in \eqref{estimation of C-2}. By the equality \eqref{C(n,D)->Q(n,D)}, we can estimate the constant $C(n,D)$ by the equalities \eqref{Q(n,D)lower bound}, \eqref{Q(n,D)upper bound} and Proposition \ref{C(1,D)}.
\begin{theo}\label{estimate of C(n,D)}
Let the constant $C(n,D)$ be as in \eqref{constant C-2}. Then we have
\begin{eqnarray*}
C(n,D)&\geqslant&\frac{1-\mathcal H_{n+1}}{n!}D^{n+1}-\frac{n-2}{2n!}D^n\log D\\
& &+\frac{1}{n!}\Biggr(\left(-\frac{1}{6}n^3-\frac{3}{4}n^2-\frac{13}{12}n+2\right)\mathcal H_n\\
& &\:+\frac{1}{4}n^3+\frac{17}{24}n^2+\left(\frac{119}{72}-\frac{1}{2}\log\left(2\pi\right)\right)n-4+\log\left(2\pi\right)\Biggr)D^n\\
& &+A_4(n,D),
\end{eqnarray*}
and
\begin{eqnarray*}
C(n,D)&\leqslant&\frac{1-\mathcal H_{n+1}}{n!}D^{n+1}-\frac{n-2}{2n!}D^n\log D\\
& &+\frac{1}{n!}\Biggr(\left(-\frac{1}{6}n^3-\frac{3}{4}n^2-\frac{13}{12}n+2\right)\mathcal H_n\\
& &\:+\frac{1}{4}n^3+\frac{17}{24}n^2+\left(\frac{119}{72}-\frac{1}{2}\log\left(2\pi\right)\right)n-4+\log\left(2\pi\right)\Biggr)D^n\\
& &+A'_4(n,D),
\end{eqnarray*}
where $n\geqslant1$, $A_4(n,D)\sim o(D^n)$, $A'_4(n,D)\sim o(D^n)$. In addition, we can calculate $A_4(n,D)$ et $A'_4(n,D)$ explicitly.
\end{theo}
\begin{proof}
First, we consider the remainders of $A_4(n,D)$ and $A'_4(n,D)$. we define $A_4(1,D)=-A_3(D)$, and
\[A_4(n,D)=\sum_{m=1}^DA_4(n-1,m)-A'_2(n,D).\]
Similarly, we define $A'_4(1,D)=-A_3(D)$, and
\[A'_4(n,D)=\sum_{m=1}^DA'_4(n-1,m)-A_2(n,D).\]
We can confirm that we have $A_4(n,D)\sim o(D^n)$ and $A'_4(n,D)\sim o(D^n)$, and they can be calculated explicitly.

Next, we are going to calculate the coefficients of $D^{n+1}$, $D^n\log D$ and $D^n$ in the estimate of $C(n,D)$. Let $a_n,b_n,c_n$ be the coefficients of $D^{n+1}$, $D^n\log D$ and $D^n$ in $C(n,D)$ respectively. By Proposition \ref{C(1,D)}, we have $a_1=-\frac{1}{2}$, $b_1=\frac{1}{2}$, $c_1=-1+\frac{1}{2}\log\left(2\pi\right)$; and by the equality \eqref{C(n,D)->Q(n,D)}, we have
\[a_n=\frac{a_{n-1}}{n+1}-\frac{\mathcal H_n}{(n+1)!}\]
and
\[b_n=\frac{b_{n-1}}{n}-\frac{1}{2n!}.\]
We consider the coefficients in the asymptotic estimate of $D^n$ in the sum $\sum\limits_{m=0}^Dm^n$ and $\sum\limits_{m=1}^Dm^{n-1}\log m$. We obtain the the asymptotic coefficient of $D^n$ in $\sum\limits_{m=1}^Dm^n$ is $\frac{n+1}{2}$, and the asymptotic coefficient of $D^n$ in $\sum\limits_{m=1}^Dm^{n-1}\log m$ is $\frac{1}{n^2}$. And the terms $A_4(n,D)$ and $A'_4(n,D)$ have no contribution to the coefficient of the term $D^n$. So we obtain
\[c_n=\frac{c_{n-1}}{n}+\frac{b_{n-1}}{n^2}+\frac{n+1}{2}a_{n-1}-\frac{1}{n!}\left(-1+\frac{1}{2}\log\left(2\pi\right)-\frac{1}{2n}+\frac{(n+1)\mathcal H_{n-1}}{2}\right).\]

For the term $a_n$, we have
\begin{equation*}
   (n+1)!a_n=n!a_{n-1}-\mathcal H_n=a_1-\sum_{k=2}^n\mathcal H_k=(n+1)(1-\mathcal H_{n+1}),
\end{equation*}
and we obtain
\[a_n=\frac{1-\mathcal H_{n+1}}{n!}.\]

For the term $b_n$, we have
\begin{equation*}
  n!b_n=(n-1)!b_{n-1}-\frac{1}{2}=b_1-\frac{n-1}{2}=-\frac{n-2}{2},
\end{equation*}
and we obtain
\[b_n=-\frac{n-2}{2n!}.\]

For the term $c_n$, by the above results of $a_n$ and $b_n$, we have
\begin{eqnarray*}
  n!c_n&=&(n-1)!c_{n-1}-\frac{n-3}{2n}+\frac{n(n+1)(1-\mathcal H_{n})}{2}\\
  & &-\left(-1+\frac{1}{2}\log\left(2\pi\right)-\frac{1}{2n}+\frac{(n+1)\mathcal H_{n-1}}{2}\right)\\
  &=&(n-1)!c_{n-1}+1-\frac{1}{2}\log\left(2\pi\right)+\frac{5}{2n}+\frac{(n+1)^2}{2}-\frac{(n+1)^2}{2}\mathcal H_{n+1}+\frac{1}{2}\\
&=&c_1-(n-1)\left(-1+\frac{1}{2}\log\left(2\pi\right)\right)+\frac{5\mathcal H_n}{2}-\frac{5}{2}+\frac{1}{12}(n+1)(n+2)(2n+3)\\
& &-\frac{5}{2}-\sum_{k=3}^{n+1}\frac{k^2}{2}\mathcal H_k+\frac{n-1}{2}\\
&=&-(n-2)\left(-1+\frac{1}{2}\log\left(2\pi\right)\right)+\frac{5\mathcal H_n}{2}+\frac{1}{6}n^3+\frac{3}{4}n^2+\frac{7}{6}n-5-\sum_{k=3}^{n+1}\frac{k^2}{2}\mathcal H_k,
\end{eqnarray*}
By the Abel transformation, we have
\begin{eqnarray*}
  \sum_{k=3}^{n+1}k^2\mathcal H_k&=&\sum_{k=1}^{n+1}k^2\mathcal H_k-7\\
  &=&\mathcal H_{n+2}\sum_{k=1}^{n+1}k^2-\sum_{k=1}^{n+1}\frac{1}{k+1}\sum_{j=1}^{k}j^2-7\\
  &=&\frac{1}{6}\mathcal H_{n+2}(n+1)(n+2)(2n+3)-\frac{1}{6}\sum_{k=1}^{n+1}k(2k+1)-7 \\
  &=&\left(\frac{1}{3}n^3+\frac{3}{2}n^2+\frac{13}{6}n+1\right)\mathcal H_{n}-\frac{1}{9}n^3+\frac{1}{12}n^2+\frac{37}{36}n-6.
\end{eqnarray*}

So we obtain
\begin{eqnarray*}
  c_n&=&\frac{1}{n!}\Biggr(\left(-\frac{1}{6}n^3-\frac{3}{4}n^2-\frac{13}{12}n+2\right)\mathcal H_n\\
& &\:+\frac{1}{4}n^3+\frac{17}{24}n^2+\left(\frac{119}{72}-\frac{1}{2}\log\left(2\pi\right)\right)n-4+\log\left(2\pi\right)\Biggr).
\end{eqnarray*}
Then we have the result.
\end{proof}
\backmatter

\bibliography{liu}
\bibliographystyle{smfplain}

\end{document}